\documentclass[a4paper, 11pt]{article}
\usepackage[utf8]{inputenc}     % Si l'encodage du fichier est unicode UTF-8
\usepackage[T1]{fontenc}
\usepackage{lmodern}
\usepackage{graphicx}
\usepackage[english]{babel}
\usepackage{subfigure}
\usepackage{amsmath, amsfonts, amssymb,amsthm,mathtools}  
\usepackage{textcomp}
\usepackage{mathrsfs}
\usepackage{multirow}
\usepackage{float,url}

\usepackage{tikz} % dessins et figures
\usetikzlibrary{arrows}
\usetikzlibrary{decorations}

\theoremstyle{plain}
\newtheorem{theorem}{Theorem}[section]

\newtheorem{lemma}[theorem]{Lemma}
\newtheorem{corollary}[theorem]{Corollary}

\newtheorem{remark}[theorem]{Remark}

\theoremstyle{definition}
\newtheorem{definition}[]{Definition}

\newcommand{\bydef}{\,\stackrel{\mbox{\tiny\textnormal{\raisebox{0ex}[0ex][0ex]{def}}}}{=}\,}

\newcommand{\R}{\mathbb{R}}
\newcommand{\C}{\mathbb{C}}

\newcommand{\tx}{\tilde{x}}
\newcommand{\tX}{\tilde{X}}
\newcommand{\bX}{\bar{X}}

\newcommand{\cM}{\mathcal{M}}
\newcommand{\cE}{\mathcal{E}}
\newcommand{\cS}{\mathcal{S}}

\newcommand{\cI}{\mathcal{I}}
\newcommand{\cont}{\mathcal{C}}
\newcommand{\Rn}{\R^n}
\newcommand{\nxn}{{n\times n}}
\newcommand{\Rnxn}{\R^\nxn}
\newcommand{\bmat}[1]{\begin{bmatrix}#1\end{bmatrix}}
\newcommand{\T}{\mathcal{T}}
\newcommand{\cN}{\mathcal{N}}
\newcommand{\Tx}[1]{\mathcal{T}_{\x{#1}}}
\newcommand{\lam}{\lambda}

\newcommand{\bx}{\mathbf{x}}
\newcommand{\bp}{\mathbf{p}}
\newcommand{\x}[1]{\mathbf{x}^{(#1)}}

\newcommand{\diag}{\operatorname{diag}}
\newcommand{\Arg}{\operatorname{Arg}}

\begin{document}
\title{
Cusp bifurcations: numerical detection via two-parameter continuation and computer-assisted proofs of existence
}

\author{Jean-Philippe Lessard\thanks {McGill University, Department of Mathematics and
Statistics, 805 Sherbrooke Street West, Montreal, QC, H3A 0B9, Canada ({\tt jp.lessard@mcgill.ca}). This author was supported by NSERC.}
\and Alessandro Pugliese \thanks{
Dipartimento di Matematica, Universit\`a degli Studi di Bari Aldo Moro,
Via E. Orabona 4, 70125 Bari, Italy ({\tt alessandro.pugliese@uniba.it}). This author gratefully acknowledges the hospitality of the School of Mathematics of Georgia Tech and the support provided by Universit\`a degli Studi di Bari Aldo Moro and by INdAM--GNCS. This author has been partially funded by  PRIN2022PNRR n. P2022M7JZW ``SAFER MESH - Sustainable mAnagement oF watEr Resources ModEls and numerical MetHods'' research grant, funded by the Italian Ministry of Universities and Research (MUR) and  by the European Union through Next Generation EU, M4.C2.1.1, CUP H53D23008930001.
%This author was supported in part by INdAM -- GNCS.
}
}

\maketitle

%\tableofcontents

\begin{abstract}
This paper introduces a novel computer-assisted method for detecting and constructively proving the existence of cusp bifurcations in differential equations. The approach begins with a two-parameter continuation along which a tool based on the theory of Poincaré index is employed to identify the presence of a cusp bifurcation. Using the approximate cusp location, Newton's method is then applied to a given augmented system (the cusp map), yielding a more precise numerical approximation of the cusp. Through a successful application of a Newton-Kantorovich type theorem, we establish the existence of a non-degenerate zero of the cusp map in the vicinity of the numerical approximation. Employing a Gershgorin circles argument, we then prove that exactly one eigenvalue of the Jacobian matrix at the cusp candidate has zero real part, thus rigorously confirming the presence of a cusp bifurcation. Finally, by incorporating explicit control over the cusp's location, a rigorous enclosure for the normal form coefficient is obtained, providing the explicit dynamics on the center manifold at the cusp. We show the effectiveness of this method by applying it to four distinct models.
\end{abstract}

\begin{center}
\end{center}

\begin{center}
{\bf \small Mathematics Subject Classification (2020)}  \\ \vspace{.05cm}
{\small 37G99 $\cdot$ 65P30 $\cdot$ 65G40 $\cdot$ 34C23 $\cdot$ 37M20 } 
\end{center}

\section{Introduction}\label{sec:intro}
%!TEX root = cusp_paper_revised.tex

A {\em cusp bifurcation} is a codimension-two bifurcation that provides an organizing center at which two saddle-node bifurcations come together. This fundamental phenomenon in bifurcation theory often gives rise to a hysteresis loop, which leads to the important notion of bistability. Despite its undeniable importance in dynamical systems theory, studying rigorously the occurrence  of a cusp bifurcation is hindered by two main factors. First, as a codimension-two bifurcation, detecting a cusp requires varying two parameters. Second, actually proving the existence of a cusp in a system of nonlinear differential equations is nontrivial, as it requires having a precise location of the cusp point and then evaluating carefully high order derivatives at the point. In this paper, we propose an approach which addresses these issues. To numerically detect a cusp, we perform a two-parameter continuation of equilibria (e.g. see \cite{MR1298047,MR1645376,MR3464215}) along which we apply the theory of Poincaré index to search for zeros of a specific two dimensional vector field (see Section~\ref{sec:detection}). Once a cusp has been numerically detected during the continuation, we apply Newton's method to a carefully chosen map (the {\em cusp map} defined in \eqref{eq:cusp_map}) to refine the approximation which is then validated using a computer-assisted proof based on Newton-Kantorovich theorem (see Section~\ref{sec:CAPs}). Finally, using the explicit control of the location of the cusp, we obtain rigorously the normal form on the center manifold. To formalize the ideas, let us introduce some notation and background. 

Consider open sets $U \subset \R^n$ and $\Lambda \subset \R^2$ and a parameter dependent vector field $f:U \times \Lambda \to \R^n$ which is $\cont^3$ in its first argument (phase space) and $\cont^1$ in its second argument (parameter space). Let
\begin{equation} \label{eq:general_vector_field}
\dot x = f(x,\lambda), \qquad x \in U \subset \R^n, \quad \lambda \in \Lambda \subset \R^2
\end{equation}
be the corresponding parameter dependent differential equation. Denote its set of equilibria by 
\begin{equation} \label{eq:solution_manifold}
\cE \bydef \{ (x,\lambda) \in U \times \Lambda : f(x,\lambda)=0 \}.
\end{equation}
Following the presentation from \cite{MR1694631,MR2071006}, we introduce the conditions for a cusp bifurcation to occur in \eqref{eq:general_vector_field}. Since a cusp bifurcation is a point in $\cE$ at which two saddle-node bifurcations merge, let us begin by recalling the ingredients necessary for a saddle-node bifurcation to occur. Assume first that $(x,\lambda) \in \cE$, that is $f(x,\lambda)=0$. At a saddle-node bifurcation, the Jacobian matrix $D_xf(x,\lambda)$ has a simple zero eigenvalue and no other eigenvalues on the imaginary axis. Denote by $v,w \in \R^n$ two non-zero vectors such that $D_xf(x,\lambda)v=0$ and $D_xf(x,\lambda)^T w=0$, normalized so that $w^T v = 1$. Here, $M^T$ denotes the transpose of a given matrix $M$. Denoting the coefficient
\begin{equation} \label{eq:coefficient_b}
b \bydef \frac{1}{2} w^T D_{xx} f(x,\lambda)(v,v) \in \R,
\end{equation}
we say that a {\em saddle-node bifurcation} occurs at $(x,\lambda)$ if $b \ne 0$. In this case, it can be shown (e.g. see \cite{MR2071006} for more details) that the restriction of the differential equation \eqref{eq:general_vector_field} at $\lambda$ to the one-dimensional center manifold has a normal form $\dot y = b y^2 + O(y^3)$, where $y\in \R$. 
At a cusp bifurcation, the coefficient $b$ in \eqref{eq:coefficient_b} vanishes (after all, it is a point at which two saddle-node bifurcation points collide), and we must resort to a higher order normal form. To make this precise, let $(h,s) \in \R^{n+1}$ be the unique solution of
\begin{equation} \label{eq:bordering_system}
\begin{pmatrix}
D_{x} f (x,\lambda) & v \\ w^T & 0
\end{pmatrix}
\begin{pmatrix}
h \\ s
\end{pmatrix}
= 
\begin{pmatrix}
- D_{xx} f (x,\lambda)(v,v) \\ 0
\end{pmatrix},
\end{equation}
where the unicity of the solution comes from the invertibility of the matrix on the left-hand side (e.g. see Lemma 5.3 in \cite{MR2071006}). From this, we get that $D_{x} f (x,\lambda) h + s v = - D_{xx} f (x,\lambda)(v,v)$, and hence
\[
w^T D_{x} f (x,\lambda) h + s w^T v = - w^T D_{xx} f (x,\lambda)(v,v)
\Leftrightarrow 0+s = -2b = 0.
\]
Then $s=0$, that is $h$ solves $D_{x} f (x,\lambda) h = - D_{xx} f (x,\lambda)(v,v)$. Letting
\begin{equation} \label{eq:nondegeneracy_coeff_intro}
c \bydef \frac{1}{6} w^T \left( D_{xxx} f (x,\lambda)(v,v,v) + 3 D_{xx} f (x,\lambda)(v,h) \right),
\end{equation}
we say that a {\em cusp bifurcation} occurs at the point $(x,\lambda)$ if $c \ne 0$ (e.g. see \cite{MR1694631,MR2071006}). Moreover, the restriction of \eqref{eq:general_vector_field} to the one-dimensional center manifold has a normal form
\begin{equation} \label{eq:cusp_normal_form_intro}
\dot y = c y^3 + O(y^4).
\end{equation}

We summarize the above discussion in the following lemma. 

\begin{lemma} \label{lem:cusp}
Let $(x,\lambda) \in \cE$ be such that $D_xf(x,\lambda)$ has one zero eigenvalue and all its remaining eigenvalues have nonzero real part. Let $v,w \in \R^n$ such that $w^T v = 1$, $D_xf(x,\lambda)v=0$ and $D_xf(x,\lambda)^T w=0$. Recall \eqref{eq:coefficient_b} and \eqref{eq:nondegeneracy_coeff_intro}, and assume that $b=0$ and $c \ne 0$. Then a cusp bifurcation occurs at $(x,\lambda)$. Moreover, the one-dimensional center manifold has normal form given by \eqref{eq:cusp_normal_form_intro}.
\end{lemma} 

This paper presents a computational methodology for rigorously verifying the hypotheses outlined in Lemma~\ref{lem:cusp}. Our approach consists of three primary steps. Initially, we identify cusps using a blend of multi-parameter continuation and Poincaré index theory. Subsequently, we conduct computer-assisted validations to confirm the presence of cusp bifurcations by verifying the conditions specified in Lemma~\ref{lem:cusp}. Finally, by exercising precise control over the cusp's positioning, we achieve a rigorous enclosure for the coefficient $c$ within the normal form \eqref{eq:cusp_normal_form_intro}. We elaborate on each of these contributions.

The process begins with the detection step (see Section~\ref{sec:detection}), which is done by computing the Poincaré index of curves (i.e. the boundaries of two-dimensional simplices used in building a triangulation of a manifold $\cM \subset \cE$) with respect to a carefully chosen two-dimensional vector field. The computation of the Poincaré index requires in turn to perform the continuation of the singular value decomposition of $D_xf(x, \lambda)$ along the boundaries of simplices. If the Poincaré index is nontrivial, then a cusp bifurcation point has been detected ``inside'' the simplex. This approach is different from the more standard cusp detection, which consists of the following two steps: (a) do a one-dimensional continuation (e.g. pseudo-arclength continuation, see \cite{MR910499}) to detect a saddle-node bifurcation (fold); and (b) free an extra parameter, perform a continuation of folds, monitor the coefficient $b$ defined in \eqref{eq:coefficient_b}, and claim that a cusp bifurcation has been detected when $b$ changes sign along the branch of folds. This well-known method can be applied using standard one-dimensional continuation softwares such as AUTO \cite{MR635945} and MATCONT \cite{MR2000880}. However, detecting a cusp bifurcation via one-dimensional continuations can be challenging because the chosen direction in parameter space may not go through a saddle-node bifurcation. As a result, even when starting the continuation near a cusp, this approach could fail at detecting it. In contrast, our approach relies on a two-dimensional continuation \cite{MR1298047,MR1645376,MR3464215}, and has the advantage of locating cusp bifurcation points on two-dimensional manifolds $\cM \subset \cE$. An obvious drawback is the computational overhead with respect to the approach based on one-dimensional continuations. However, a clear upside is that we construct explicitly a portion of the two-dimensional manifold $\cM \subset \cE$, as opposed to just one dimensional branches. In short, we get more at a larger cost.

Once a cusp is numerically detected, the numerical candidate for the cusp is refined by applying Newton's method to the {\em cusp map} $F:\R^{4n+4} \to \R^{4n+4}$ defined in \eqref{eq:cusp_map}. Then, a computer-assisted proof (see Section~\ref{sec:CAPs}) is performed by applying a Newton-Kantorovich Theorem (see Theorem~\ref{thm:NK}) to prove existence of a zero of $F$. By construction, a zero $(x,v,w,\lambda,h,s) \in \R^{4n+4}$ of $F$ satisfies $f(x,\lambda) = D_xf(x,\lambda) v =D_xf(x,\lambda)^T w = 0$ and $b=0$. Moreover, the non-degeneracy of the zero of $F$ (a consequence of the Newton-Kantorovich Theorem) implies directly that $c \ne 0$, with $c$ given in \eqref{eq:nondegeneracy_coeff_intro}. Finally, the explicit control on the location of $(x,\lambda)$ is used to show that $0$ is the only eigenvalue of $D_xf(x,\lambda)$ with zero real part. From Lemma~\ref{lem:cusp}, a cusp bifurcation occurs at $(x,\lambda)$ and the normal form on the center manifold is given by \eqref{eq:cusp_normal_form_intro}. 

Before proceeding any further, it is important to note that employing computer-assisted proofs to investigate bifurcations in both finite and infinite dimensional dynamical systems is not a new concept. This approach has been explored extensively in various domains. For instance, within the realm of ordinary differential equations (ODEs), saddle-node bifurcations have been extensively studied within vector fields \cite{Kanzawa_Oishi,MR3542954}, and in regulatory network dynamics \cite{shane_elena2023}. Additionally, different types of bifurcation points such as symmetry-breaking \cite{Kanzawa_Oishi,MR3542954}, double turning points \cite{MR2059468,MR2319947}, period-doubling \cite{MR2534406}, and cocoon bifurcations \cite{MR2351028} have been investigated. The exploration extends to partial differential equations (PDEs) where saddle-node and symmetry-breaking bifurcations of steady states have been explored \cite{MR2679365,MR3808252,MR3792794,MR3390404}. Furthermore, Hopf bifurcations have been rigorously analyzed with computer-assisted proofs not only in ODEs \cite{MR4238006} but also in neural network dynamics \cite{kuehn_queirolo}, delay differential equations (DDEs) \cite{MR4337868}, and PDEs \cite{MR4372114}. Recent research has delved into specific aspects such as Hopf bubbles and degenerate Hopf bifurcations \cite{church_queirolo}, and solving conjectures like Wright's conjecture \cite{MR3779642} through normal form computations of Hopf bifurcations in delay equations. Finally, there is a growing interest in studying center manifolds using rigorous numerical methods as evidenced by recent literature \cite{jb_center_manifold1,jb_center_manifold2,MR4217108,MR3397322,MR3567489,MR4664056}.

%Before proceeding any further, it is worth mentioning that the idea of using computer-assisted proofs to study bifurcations in finite and infinite dimensional dynamical systems is by no means new. Indeed, in the field of ODEs, saddle-nodes are considered in vector fields \cite{Kanzawa_Oishi,MR3542954} and in regulatory network dynamics in \cite{shane_elena2023}. Moreover, symmetry breaking bifurcation points are considered in \cite{Kanzawa_Oishi,MR3542954}, double turning points in \cite{MR2059468,MR2319947}, period-doubling bifurcations in \cite{MR2534406} and cocoon bifurcations in \cite{MR2351028}. Saddle-node and symmetry breaking bifurcations of steady states in PDEs are considered in \cite{MR2679365,MR3808252,MR3792794,MR3390404}. Moreover, Hopf bifurcations in ODEs \cite{MR4238006}, in neural network dynamics \cite{kuehn_queirolo}, in DDEs \cite{MR4337868} and in PDEs \cite{MR4372114} have been studied with rigorous numerics. Recently, Hopf bubbles and degenerate Hopf bifurcations were studied in \cite{church_queirolo} and Wright's conjecture was solved in \cite{MR3779642}  by computing a normal form computation of the Hopf bifurcation in the delay equation. Lastly, center manifolds are also starting to be studied with the tools of rigorous numerics (e.g. see \cite{jb_center_manifold1,jb_center_manifold2,MR4217108,MR3397322,MR3567489,MR4664056}).

Our paper is structured as follows. In Section~\ref{sec:detection}, we outline the algorithm employed for constructing a numerical approximation of a segment of the equilibria manifold $\cE$ for equation \eqref{eq:general_vector_field}, followed by our methodology for numerically identifying cusp points on this manifold. Section~\ref{sec:CAPs} is dedicated to presenting computer-assisted proofs utilizing a Newton-Kantorovich argument to establish the existence of cusp bifurcations. In Section~\ref{sec:applications}, we conclude our study by applying our methodology to four distinct ODE models: Bazykin’s model, a predator-prey system featuring a non-monotonic response function, a model for metastatic cell transitions, and Bykov’s model.

\section{Two-parameter continuation and numerical detection of cusps} \label{sec:detection}
%!TEX root = cusp_paper_revised.tex

In this section we first describe the algorithm that we use to construct a numerical approximation of (a portion of) the set $\cE$ of equilibria for \eqref{eq:general_vector_field}, and then present our method for the numerical detection of cusp points. 

\subsection{Two-parameter continuation}\label{subsec:2Dcont}
The continuation algorithm, in its current form, has been introduced in \cite{MR3464215}, where the authors have used it in the context of rigorous continuation of families of equilibria of multi-parameter infinite dimensional problems. It has been used also in \cite{CoDeLoPu} for computing separatrices between regions that lead to crossing and sliding behaviour in non-smooth dynamical systems, and in \cite{church_queirolo} to compute two-dimensional surfaces of periodic orbits in ODEs giving rise to Hopf bubbles via degenerate Hopf bifurcations. Here we will give a short account on how the algorithm works, and refer the reader to \cite{MR3464215} for a detailed description. 

Following the notations introduced in the previous section, we consider $f:U\times\Lambda\to\Rn$ sufficiently smooth, with $U\subset\Rn$ and $\Lambda\subset\R^2$ open, and the set of equilibria $\cE$ defined in \eqref{eq:solution_manifold}. For convenience, in this section we will often group $x\in\Rn$ and $\lam\in\R^2$ together into a ${(n+2)}$-dimensional vector: $\mathbf{x}=(x,\lam)$. Under the generic assumption of $0$ being a regular value for $f$, $\cE$ is, at least locally, a two-dimensional manifold embedded in $\R^{n+2}$, and is as smooth as $f$. The aim of the algorithm is to construct a numerical triangulation of a portion of $\cE$ under interest, that we denote by $\cM\subset\cE$. The set $\cM$ may have one connected component or be the disjoint union of several connected components. We assume for simplicity that $\cM$ has only one component, as, otherwise, everything that follows applies without distinction to each component. Upon completion, the algorithm's output is a finite collection of nodes on $\cM$ and simplices (in fact triangles) each having three of those nodes as vertices, together with information about the adjacency relations between simplices. 

The algorithm follows the general philosophy of predictor-corrector methods. It starts off from an initial point $\x{0}$ on $\cM$, it computes a first patch\footnote{By ``patch centered at $\bx$'' we mean the collection of simplices that share the node $\bx\in\cM$. } on $\cM$ centered at $\x{0}$, and then proceeds through an advancing front strategy, progressively adding new nodes $\x{k}$ and simplices to the triangulation until $\cM$ has been completely approximated. See Figure \ref{fig:figure1} for a visual description.

\begin{figure}[h]
\centering
\includegraphics[width=.85\textwidth]{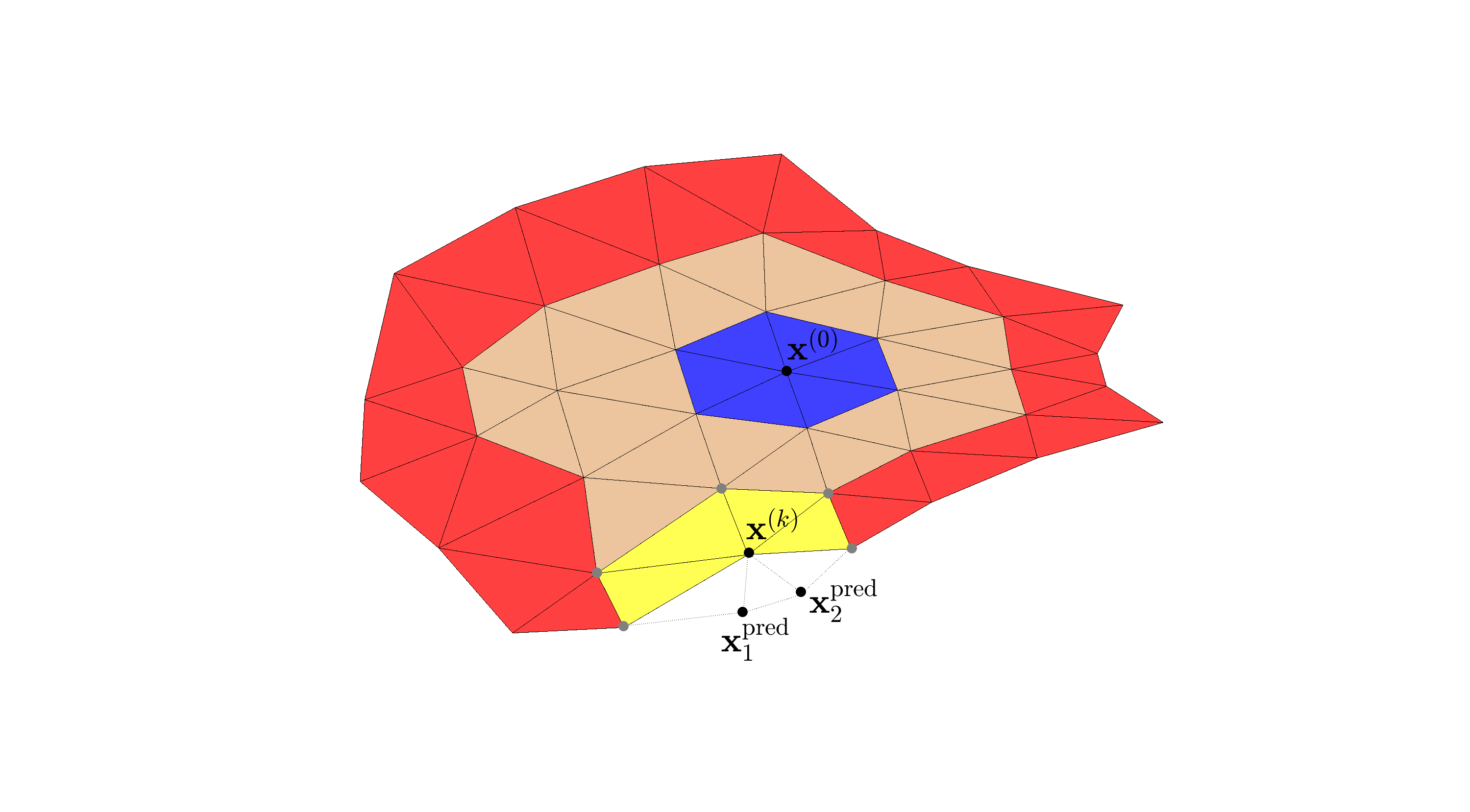}\label{fig:figure1}
\caption{Visual sketch of the continuation procedure: in blue the first hexagonal patch, in yellow an incomplete patch centered at the ``current point'' $\x{k}$, in red/yellow the simplices which are on the boundary of the (un)explored portion of $\cM$.}
\end{figure}

We now describe first the general step, and then the initialization.

\medskip
\noindent{\bf General step.} Let $\x{k}$ be the ``current point'' on $\cM$, center of an incomplete\footnote{A patch is said to be complete if, upon projection on the tangent plane to $\cM$ at $\bx$, the sum of the angles of its simplices at $\bx$ is $2\pi$.} patch (see Figure \ref{fig:figure1}). Prediction steps are taken on the tangent plane $\Tx{k}\cM$ to $\cM$ at $\x{k}$ as follows: the external nodes of the incomplete patch centered at $\x{k}$ (the grey points in Figure \ref{fig:figure1}) are projected on $\Tx{k}\cM$, and then predictions $\mathbf{x}^\mathrm{pred}_j$ are formed on $\Tx{k}\cM$ so that the projected polygonal patch on $\Tx{k}\cM$ centered at $\x{k}$ is complete. The correction step consists of projecting the predictions $\mathbf{x}^\mathrm{pred}_j$ back to $\cM$. All projections to and from $\cM$ are numerically performed by applying the Gauss-Newton method along the direction perpendicular to $\Tx{k}\cM$. Stepsize (i.e. the distance between the center $\x{k}$ and the predictions) is adaptively chosen based on a user provided tolerance, and will be inversely proportional to the local curvature of the manifold. 

\medskip
\noindent{\bf Initialization.} As in any continuation method, the algorithm requires appropriate initialization. This consists of an initial point $\x{0}\in\cM$ provided by the user and tolerances required for stepsize selection, convergence of Gauss-Newton iterates, and boundary of the portion of manifold to be explored. The algorithm starts off by projecting on $\cM$ the six vertices of a regular hexagon lying on $\T_{\x{0}}\cM$ and centered at $\x{0}$, producing the first patch, and then proceeds through the general step. Again, see Figure \ref{fig:figure1}.

\subsection{Detection of cusp bifurcations}\label{subsec:detect}

In this section we illustrate our strategy for the detection of a cusp bifurcation on $\cM$. Our main result will be Theorem \ref{thm:mainDetect}. To prove this theorem, we first need to review some fundamental results which constitute its two main ingredients: the theory of Poincaré index and the smooth singular value decomposition of matrix valued functions. We do so in Sections \ref{subsubsec:poincare} and \ref{subsubsec:smoothSvd} below.

\subsubsection{Poincaré index}\label{subsubsec:poincare} The Poincaré index is a topological tool that can reveal the presence of isolated zeros of a continuous map inside the region bounded by a Jordan curve. Its definition goes back to the original work of Poincaré \cite{H1881}. Our exposition is adapted from \cite{MR0069338}.%, Chapter 16, Section 4. 
\begin{definition}[Poincaré index]\label{def:poincare} Let $g:(s_1,s_2)\in D\mapsto g(s_1,s_2)\in\R^2$ be continuous, with $D$ a bounded open subset of $\R^2$, and suppose $g$ has only isolated zeros in $D$. Let $\Gamma\subset D$ be a Jordan curve passing through no zeros of $g$. Let $\Delta\theta$ be the total change\footnote{Let $\Gamma$ be a Jordan curve that does not contain the origin, and let it be parametrized by a periodic function $\gamma(t)$ with minimum period 1. Extend the notion of the argument of a complex number $z$, denoted by $\Arg(z)$, to a point on the plane in the natural way, and note that $\Arg$ is a multivalued map. Let $\theta$ be any continuous selection of the map  $t \in [0,1] \mapsto \Arg(\gamma(t))$. Then, the ``total change in the angle $\theta$ that $\gamma(t)$ makes with respect to any fixed direction in $\mathbb{R}^2$ as it traces $\Gamma$ once with positive orientation'' is defined as $\theta(1) - \theta(0)$. A continuous selection exists because $\Gamma$ is compact and does not contain the origin. In fact, there are infinitely many, but the value of $\theta(1) - \theta(0)$ does not depend on the choice of the selection.} in the angle $\theta$ that $g(s_1,s_2)$ makes with respect to some fixed direction of $\R^2$ as $(s_1,s_2)$ traces $\Gamma$ once with positive (i.e. counterclockwise) orientation. Then, the Poincaré index of $\Gamma$ with respect to $g$ is defined as $\Delta\theta/2\pi$, and will be denoted by $\cI_g(\Gamma)$. 
\end{definition}

Clearly, the number $\cI_g(\Gamma)$ is an integer. Its importance lies in the following two fundamental results, which also require $D$ to be simply connected.
\begin{theorem}\label{thm:poincare} Let $D$ be open, bounded, and simply connected. If $\Gamma\subset D$ is a Jordan curve containing no zeros of $g$ on it or in its interior, then $\cI_g(\Gamma)=0$.
\end{theorem}
\begin{corollary}\label{cor:poincare} Let $g:(s_1,s_2)\in D\mapsto g(s_1,s_2)\in\R^2$ be continuous, where $D\subset\R^2$ is open, bounded, and simply connected. Let $\Gamma\subset D$ be a Jordan curve passing through no zeros of $g$. If $\cI_g(\Gamma)\ne0$, then $g$ has a zero inside the region bounded by $\Gamma$.
\end{corollary}

\subsubsection{Singular value decomposition of matrix valued functions}\label{subsubsec:smoothSvd} The singular value decomposition (SVD) is largely regarded as the most reliable tool to gain information on the rank of a matrix and on its fundamental subspaces. The idea to use it for detecting bifurcations in dynamical systems is not new. It has originally appeared in \cite{chow1988bifurcations}, and has been used in \cite{dieci2006path} to detect folds and branch points. The authors of \cite{dieci2006path} clearly point out that the key to for making this idea effective is \emph{smoothness}. 

Below we give a brief review of facts that concern the singular value decomposition (SVD) of matrix valued functions of one or several parameters which are relevant to our work. All the results are adapted from \cite{Dieci1999OnSD} and \cite{MR2530255}, to which we refer for details.

\bigskip\noindent{\bf One parameter.} We begin with the following fundamental result:
\begin{theorem}[Adapted from \cite{Dieci1999OnSD}, Theorem 3.6]\label{thm:1parSmoothSVD} Let $t\in\R\mapsto A(t)\in\Rnxn$ be $\cont^k$, with $k\ge 0$. If $A(t)$ has distinct singular values for all $t\in\R$, than its singular values and singular vectors can be defined to be $\cont^k$ functions of $t$.  
\end{theorem}
It is important to realize that the assumption of Theorem \ref{thm:1parSmoothSVD} is not at all restrictive but, rather, what one should generally expect to hold for a one parameter matrix valued function. In fact, one can show (again, see \cite{Dieci1999OnSD}) that the set of real $\nxn$ matrices having repeated singular values has dimension $n^2-2$, and therefore a generic\footnote{We refer to \cite{hirsch_diffTopology} for necessary background on the concept of genericity.} $\cont^k$ path $t\in\R\mapsto A(t)\in\Rnxn$ will not meet it. On the other hand, the set of real $\nxn$ matrices having rank equal to $n-1$ can be easily seen to have dimension $n^2-1$, and therefore one should expect a $\cont^k$ path $A(t)$ of $\nxn$ real matrices to lose rank (from $n$ to $n-1$) at isolated points. At such points, a singular value of $A(t)$ will become zero, in which case it generally needs to be allowed to become negative to retain its smoothness (and that of the corresponding singular vector). This is why the SVD resulting from Theorem \ref{thm:1parSmoothSVD} is called  \emph{signed} SVD.

\bigskip\noindent{\bf Two parameters.} The dimension argument used above shows that one \emph{should} expect matrix valued functions of two real parameters to have coalescing singular values at isolated points in parameters' space. In addition, one can show (see \cite{MR2530255}) that generic coalescence of singular values prevents a matrix valued function of two parameters from having even continuous singular vectors in a neighbourhood of the point of coalescence. If the singular values remain always distinct, however, singular values and singular vectors can be defined so to retain the same degree of smoothness of the matrix function. We formalize this fact below.

\begin{theorem}[Adapted from \cite{MR2530255}, Theorem 2.10 and Remark 2.11]\label{thm:2parSmoothSVD} Let $(s_1,s_2)\in\Omega\mapsto A(s_1,s_2)\in\Rnxn$ be $\cont^k$, with $k\ge 0$ and $\Omega\subset\R^2$ a closed rectangle. If the singular values of $A(s_1,s_2)$ are distinct for all $(s_1,s_2)\in\Omega$, then its singular values and singular vectors can be defined to be $\cont^k$ functions of $(s_1,s_2)$.
 \end{theorem}
 
%Also the SVD of Theorem \ref{thm:2parSmoothSVD} is a signed SVD, i.e. one must allow singular values to change sign to retain overall smoothness.
Like the SVD in Theorem \ref{thm:1parSmoothSVD}, the SVD in Theorem \ref{thm:2parSmoothSVD} is also a signed SVD; that is, the singular values of \( A(s_1, s_2) \) must be allowed to become negative to preserve their smoothness and that of the singular vectors.

\subsubsection{A topological test for cusp bifurcations}\label{subsubsec:cuspDetect} We are now in the position for stating our main result about detection of cusp bifurcations. Our ultimate goal is to define, by considering the continuous SVD of the Jacobian matrix $D_xf(x,\lambda)$ on (portions of) $\cM$, a continuous map $g$ that has a zero whenever a cusp bifurcation occurs for \eqref{eq:general_vector_field}, and use the Poincaré index as a tool to detect such zeros. Although in Sections \ref{subsec:numericalDetection} and \ref{sec:applications} we will apply our method to simplices of triangulations produced by the continuation algorithm described in Section \ref{subsec:2Dcont}, here the method can (and will be) described as a stand-alone tool for detecting cusp bifurcation points on a simply connected region of a 2-manifold.
 
Consider a portion $\cN$ of $\cM$ which is the embedded image of a closed, bounded, simply connected region $\Omega\subset\R^2$, parametrized in terms of two parameters $(s_1,s_2)\in\Omega$:
\begin{equation}\label{eq:paramN}
(s_1,s_2)\in\Omega\mapsto\bx(s_1,s_2)\in\cN,
\end{equation}
and such that $\bx(\cdot)$ maps $\partial\Omega$ homeomorphically onto $\partial\cN$. Recall (see Lemma \ref{lem:cusp}) that necessary conditions for having a cusp bifurcation at $(x,\lam)$ are:
%\jp{Below, why not use the notation $D_xf(x,\lambda)$ instead of $\evalat{D_xf}_{(x,\lam)}$?}
\begin{itemize}
 \item[(i)] $D_xf (x,\lam)$ has a 1-dimensional kernel;
 \item[(ii)] $w^T D_{xx}f(x,\lam)(v,v)=0$, where $D_xf (x,\lam)v=0$ and $D_xf (x,\lam)^Tw=0$, with $w$ and $v$ non-zero vectors in $\R^n$.
\end{itemize}
These conditions can be easily rephrased in terms of the singular value decomposition of $D_xf (x,\lam)$. In fact, let $D_xf (x,\lam)=W\Sigma V^T$ be a SVD, with $\Sigma=\diag(\sigma_1,\ldots,\sigma_n)$, $\sigma_1\ge\ldots\ge\sigma_n\ge 0$ and $W=[w_1,\ldots,w_n]$, $V=[v_1,\ldots,v_n]$ real orthogonal and partitioned by columns. Then, necessary conditions for a cusp bifurcation to occur at $(x,\lam)$ are:
\begin{equation}\label{eq:cuspConditions}
\begin{split}
\text{(i)} &\quad \sigma_n=0 \text{ and } \sigma_{n-1}>0\,;\\
\text{(ii)} &\quad w_n^T D_{xx}f (x,\lam)(v_n,v_n)=0\,.
\end{split}
\end{equation}
%\begin{itemize}
% \item[(i)] $\sigma_n=0$ and $\sigma_{n-1}>0$;
% \item[(ii)] $u_n^T \evalat{D_{xx}f}_{(x,\lam)}(v_n,v_n)=0$.
%\end{itemize}
Now, consider the embedding $\bx(s_1,s_2)$ in \eqref{eq:paramN}, and define the following function
\begin{equation}\label{eq:matfunDetect}
(s_1,s_2)\in\Omega\mapsto J(s_1,s_2) \bydef D_xf (\bx(s_1,s_2))\in\Rnxn.
\end{equation}
Note that $J$ is a continuous matrix valued function of $(s_1,s_2)\in\Omega$. If $J$ has distinct singular values everywhere in $\Omega$, then it admits, by virtue of Theorem \ref{thm:2parSmoothSVD}, a continuous signed SVD:
\begin{equation}\label{eq:2parSignedSVD}
 J(s_1,s_2)=W(s_1,s_2)\Sigma(s_1,s_2)V^T(s_1,s_2),\quad (s_1,s_2)\in\Omega,
\end{equation}
with $\Sigma(s_1,s_2)=\diag(\sigma_1(s_1,s_2),\ldots,\sigma_n(s_1,s_2))$, $\sigma_1(s_1,s_2)>\ldots>\sigma_{n-1}(s_1,s_2)>\sigma_n(s_1,s_2)$, and $W(s_1,s_2), V(s_1,s_2)$ real orthogonal and partitioned by columns, as above. Under the last assumption, the following map
\begin{equation}\label{eq:gMapDef}
(s_1,s_2)\in\Omega\mapsto g(s_1,s_2)=\bmat{g_1(s_1,s_2) \\ g_2(s_1,s_2)} \bydef 
\bmat{\sigma_n(s_1,s_2) \\ w_n(s_1,s_2)^TD_{xx}f(\bx(s_1,s_2))(v_n(s_1,s_2),v_n(s_1,s_2))}\in\R^2
\end{equation}
is continuous on $\Omega$.

We are now ready to state the main result of this section. In a nutshell, the next theorem yields the following  criterion for detecting points where a cusp bifurcation occurs on $\cN$: compute the Poincaré index of the boundary $\partial\cN$ of $\cN$ with respect to $g$, and check if it is non-zero.
\begin{theorem}\label{thm:mainDetect} Let $\bx$, $\Omega$ and $\cN$ be as in \eqref{eq:paramN}. Let $J$ be as in \eqref{eq:matfunDetect}, and suppose it has distinct singular values for all $(s_1,s_2)\in\Omega$. Let $g$ be as in \eqref{eq:gMapDef}. Consider the Poincaré index $\cI_g(\partial\cN)$. If $\cI_g(\partial\cN)\ne 0$, then the necessary conditions \eqref{eq:cuspConditions} for a cusp bifurcation to occur at a point on $\cN$ are satisfied.
\end{theorem}
\begin{proof}
The statement is an immediate consequence of Corollary \ref{cor:poincare} and of how the map $g$ is defined through \eqref{eq:gMapDef}.
\end{proof}
\begin{remark} Theorem \ref{thm:mainDetect} gives a necessary condition for a cusp bifurcation to occur; in order to make it also sufficient, one should check that the non-degeneracy condition $c\ne 0$ of Lemma \ref{lem:cusp} is satisfied, and that only one eigenvalue of $D_xf(x,\lambda)$ has zero real part.
\end{remark}
\begin{remark} In preparation for (and in the statement of) Theorem \ref{thm:mainDetect}, we made several assumptions that are stronger than     necessary, or may appear limiting. We did so mostly to ease the presentation. We elaborate on this below. 
\begin{enumerate}
 \item In general, there is no reason to expect $J(s_1,s_2)$ to have distinct singular values on $\Omega$ (see the discussion in Section \ref{subsubsec:smoothSvd}), but this assumption can be loosened, and the effects of $J$ failing to meet the loosened assumptions dealt with in an effective way. Continuity of the map \eqref{eq:gMapDef} relies on continuity of $\sigma_n(s_1,s_2)$, $w_n(s_1,s_2)$ and $v_n(s_1,s_2)$. This would be guaranteed by the sole requirement that $\sigma_{n-1}(s_1,s_2)>\sigma_n(s_1,s_2)$ everywhere on $\Omega$. Furthermore, generic coalescence of $\sigma_n$ and $\sigma_{n-1}$ at one point in $\Omega$ is associated to a sign change of $w_n$ and $v_n$, and therefore of $g_2$ in \eqref{eq:gMapDef} upon tracing a sufficiently small loop around the point of coalescence (see \cite{MR2530255}, Section 3); if detected, this situation can be handled by splitting $\cN$ into disjoint simply connected portions so that cusp bifurcation and coalescence for the pair of singular values $(\sigma_{n-1},\sigma_n)$ do not occur in the same portion; we expect this to be possible since having a cusp bifurcation at a point of coalescence is highly non-generic for functions of two parameters (see \cite{MR2530255}, Remark 2.16).%\ap{Comment on how we deal with this in practice in a ``numerics (sub)section"?} \ldots
\item Since \eqref{eq:2parSignedSVD} is a signed SVD, one cannot exclude that a loss of rank of $J$ (potentially associated to a cusp bifurcation) occurs due to some $\sigma_k$ becoming zero, with $k<n$. To deal with this, one should define one map $g$ for each of the singular values of $J$, in a similar way to what we have done for $\sigma_n$, and make the obvious adjustments to \eqref{eq:cuspConditions} and Theorem \ref{thm:mainDetect}. In practice, one can progressively do this for the singular values $\sigma_n, \sigma_{n-1}, \sigma_{n-2}, \ldots$ that do change sign along $\partial\cN$.%\ap{Postpone this last comment to a ``numerics (sub)section"?}
\end{enumerate}
\end{remark}

\subsection{Numerical implementation of the topological test}\label{subsec:numericalDetection}
In this section we describe how the topological test of Section \ref{subsubsec:cuspDetect} is implemented numerically and integrated with the continuation algorithm described in Section \ref{subsec:2Dcont}.

%Let $\cN\subset\cM$ be the embedded image of a simplex of a triangulation of $\cM$, parametrized with respect to two parameters $(s_1,s_2)\in\Omega$:
%\begin{equation}\label{eq:simplexEmbedding}
%\bp(s_1,s_2)\in\cS\mapsto\bx(s_1,s_2)\in\cN,
%\end{equation}
Let $\cS$ be a simplex of a triangulation of $\cM$, parametrized with respect to two parameters $(s_1,s_2)\in\Omega\subset\R^2$,
\begin{equation*}%\label{eq:simplexEmbedding}
(s_1,s_2)\in\Omega\mapsto\bp(s_1,s_2)\in\cS,
\end{equation*}
and let $\cN\subset\cM$ be the embedded\footnote{One possible embedding is the orthogonal projection from $\cS$ to $\cM$.} image of $\cS$
\begin{equation}\label{eq:simplexEmbedding}
\bp(s_1,s_2)\in\cS\mapsto\bx(s_1,s_2)\in\cN,
\end{equation}
such that $\bx(\cdot)$ maps $\partial\cS$ homeomorphically onto $\partial\cN$. From now on, both $\cS$ and $\cN$ will be referred to as \emph{triangles}. Since $\partial\cN$ is a Jordan curve, we can think of it as the image of a continuous map
\begin{equation*}
\gamma:t\in\R\mapsto \gamma(t)\in\partial\cN,
\end{equation*}
that we can take to be periodic of minimum period 1. Recall the definitions of $J$ and $g$ in \eqref{eq:matfunDetect} and \eqref{eq:gMapDef}, respectively. With the assumptions of Theorem \ref{thm:mainDetect}, our task is to numerically compute the index $\cI_g(\cN)$. We do so through the following procedure:
\begin{enumerate}
\item Compute a continuous signed SVD of $J$ on $\partial\cN$, represented by
\begin{equation}\label{eq:signedSVDboundary}
 J(t)=W(t)\Sigma(t)V^T(t),\quad t\in[0,1].
\end{equation}
This we do using the algorithm\footnote{The algorithm has been implemented in Matlab and is freely available on Matlab Central File Exchange at \url{https://www.mathworks.com/matlabcentral/fileexchange/160881-smooth-singular-value-decomp-of-a-real-matrix-function}} described in \cite{DIECI20081255} and \cite{DIECI2011996}, where it has been used to locate points where singular values of a real matrix valued function of two parameters coalesce. Its output is a finite sequence of diagonal matrices $\Sigma_0,\Sigma_1,\ldots,\Sigma_N$ and orthogonal matrices $W_0, W_1, \ldots, W_N, V_0, V_1, \ldots, V_N$ such that
\begin{equation*}
 \Sigma_k\approx\Sigma(t_k),\ W_k\approx W(t_k),\ V_k\approx V(t_k),\ \text{for } k=0,1,\ldots, N,
\end{equation*}
where $0<t_0<t_1<\ldots<t_N=1$ is a partition of the interval $[0,1]$ and ``$\approx$'' means ``as close as machine precision allows''. The fineness of the partition is function of a user-defined tolerance.
\item Use the output of step 1 above to evaluate the function $g$ at all points of the partition $t_0<\ldots<t_N$. This step produces a finite sequence of 2-dimensional vectors 
\begin{equation*}
\left\{\bmat{g_{1,0} \\ g_{2,0}}, \bmat{g_{1,1} \\ g_{2,1}}, \ldots, \bmat{g_{1,N} \\ g_{2,N}}\right\}
\end{equation*}
where $g_{1,k}\approx g_1(t_k)$ for $k=0, \ldots, N$, and similarly for $g_2$. 
\item Build the set of indices $I_1=\{k=0,\ldots, N-1 \text{ such that } g_{1,k}\,g_{1,k+1}<0\}$. If $I_1=\emptyset$, flag $\cN$ as \verb"NO CUSP" triangle and exit the procedure. If $I_1\ne\emptyset$, proceed to next step.
\item Compare $g_{2,0}$ with $g_{2,N}$. If $g_{2,0}\approx -g_{2,N}$, flag $\cN$ as \verb"UNDETERMINED" triangle and exit the procedure. Otherwise, proceed to next step.

\item Build the set of indices $I_2=\{k=0,\ldots, N-1 \text{ such that } g_{2,k}\,g_{2,k+1}<0\}$. If $I_2=\emptyset$, flag $\cN$ as a \verb"NO CUSP" triangle and exit the procedure. If $I_2\ne\emptyset$, proceed to next step.
\item If any of $I_1$ or $I_2$ consists of more than two indices, flag $\cN$ as \verb"UNDETERMINED" triangle and exit the procedure. Otherwise, proceed to next step. 
\item If each of $I_1$ and $I_2$ consists of exactly two indices, say $I_1=\{\iota_1, \iota_2\}$ and $I_2=\{\kappa_1, \kappa_2\}$, and they interlace, i.e. either $\iota_1<\kappa_1<\iota_2<\kappa_2$ or $\kappa_1<\iota_1<\kappa_2<\iota_2$, then flag $\cN$ as \verb"CUSP" triangle and exit the procedure. If $I_1$ and $I_2$ have an index in common, flag $\cN$ as \verb"UNDETERMINED" triangle. Otherwise, flag $\cN$ as \verb"NO CUSP" triangle and exit the procedure.
\end{enumerate}

\begin{remark}\
\begin{enumerate}
\item It is easy to see that, if $g$ is continuous, $I_1$ and $I_2$ must have a even number of elements. Moreover, the interlacing property of the indices in step 6 above guarantees that the vector $g$ visits each of the four quadrants of the Cartesian plane exactly once, either in counterclockwise or in clockwise direction, while $\cN$ is traced once; this means that $\cI_g(\cN)=\pm 1$.
\item If the procedure flags $\cN$ as \verb"CUSP" triangle, then the computer-assisted proof of Section \ref{sec:CAPs} is performed on $\cN$. 
\item Exit at step 4, signifies that the triangle may potentially contain a cusp point --and-- a point of coalescence involving the singular value $\sigma_n$, in which case $g$ fails to be continuous and the test is inconclusive. Exit at step 6 may be associated to disparate situations, from none to several cusp points. In all such situations, one may either successively refine $\cN$ to isolate possible cusps or coalescing points, or simply perform the computer-assisted proof on it. This has occurred to us in a very limited number of cases, in which we have chosen the second option, that is to run the proof. See Section \ref{sec:applications}. %and we experimented with both options, see Section \ref{sec:applications}. % have chosen the second option, that is to run the proof.% Its failure can be interpreted as a \emph{false positive} for our detection procedure.
\item Exit at step 7 with an \verb"UNDETERMINED" flag signals that, along the boundary of a triangle, two points where $g_1$ and $g_2$ (see Equation \eqref{eq:gMapDef}) change sign are so close that it is hard to infere anything about the interlacing between indices in $I_1$ and $I_2$. This event has occurred once to us, when studying the model in Section \ref{sec:metastatic}. We resolved it by adjusting the tolerances for the continuation of the triangulation and of the continuous SVD in Equation \eqref{eq:signedSVDboundary}, but one could also handle it in the same way as suggested in 3.\ above.
\item In principle, one should also verify that the hypotheses of Corollary \ref{cor:poincare} are satisfied. However, we did not do this, as a cusp occurring on one side of $\cN$ is clearly a non-generic event. Furthermore, such an event is automatically detected in step 7 of the detection procedure and thus addressed in point 4 of this remark.
\end{enumerate}
\end{remark}

\section{Computer-assisted proofs for cusp bifurcations} \label{sec:CAPs}
%!TEX root = cusp_paper_revised.tex

In this section, we introduce the cusp map $F:\R^{4n+4} \to \R^{4n+4}$ whose zeros provide natural candidates for cusps. More precisely, we demonstrate a result in Theorem~\ref{thm:cusp} which states that non-degenerate zeros of $F$ satisfying an extra hypothesis on the eigenvalues of $D_xf(x,\lambda)$, lead to cusp bifurcation points. This is done in Section~\ref{sec:cusp_map}. Then, in Section~\ref{sec:NKT}, we introduce a Newton-Kantorovich Theorem, whose successful application leads to the existence of non-degenerate zeros of $F$. We continue in Section~\ref{sec:spectrum_enclosure} where we introduce a technique based on the Gershgorin Circle Theorem to control all eigenvalues of $D_xf(x,\lambda)$. Finally, we conclude in Section~\ref{sec:procedure_cusp} by presenting a procedure for proving constructively existence of cusp bifurcations.

\subsection{The cusp map and the property of its non-degenerate zeros} \label{sec:cusp_map}

As mentioned earlier, a cusp is a codimension-two bifurcation, and therefore requires solving for two parameters for the phenomenon to occur. The unknowns we solve for are $x$ (the state), the two-dimensional parameter $\lambda \in \R^2$, the vector $v$ in the kernel of $D_xf(x,\lambda)$, the vector $w$ in the kernel of $\left( D_xf(x,\lambda) \right)^T$, and the solution $(h,s)$ of the bordering system \eqref{eq:bordering_system}. Hence, denoting the vector of unknowns by $X \bydef (x,v,w,\lambda,h,s) \in \R^{4n+4}$, define the {\em cusp map} $F:\R^{4n+4} \to \R^{4n+4}$ acting on $X$ as 
\begin{equation} \label{eq:cusp_map}
F(X) = F(x,v,w,\lambda,h,s) \bydef \begin{pmatrix} 
f(x,\lambda) 
\\ 
D_xf(x,\lambda)v + s_1 v
\\
D_xf(x,\lambda)^T w
\\  
v^T v -1 
\\ 
w^T v - 1
\\
D_{x} f (x,\lambda) h + s_2 v + D_{xx} f (x,\lambda)(v,v)
\\
w^T D_{xx} f (x,\lambda)(v,v) 
\\
w^T h
\end{pmatrix}.
\end{equation}

\begin{remark}
The cusp map \eqref{eq:cusp_map} is similar to the augmented map defined in equation (5.2) in \cite{MR787206} (later also used in \cite{Kanzawa_Oishi}) for identifying cusp bifurcation points. However, the authors there did not include the last two equations present in \eqref{eq:cusp_map}, which we do because we want to have access directly to the solution $(h,s)$ of \eqref{eq:bordering_system} necessary to compute the normal form coefficient $c$ given in \eqref{eq:nondegeneracy_coeff_intro}. 
\end{remark}

As the following result demonstrates, a non-degenerate zero $X=(x,v,w,\lambda,h,s)$ of the cusp map has the property (modulo verifying an extra assumption on the eigenvalues of $D_x f(x,\lambda)$) that $(x,\lambda)$ is a cusp bifurcation.  

\begin{theorem} \label{thm:cusp}
Consider a smooth (at least $\cont^3$) map $f:\R^{n+2} \to \R^n$. Assume that $X = (x,v,w,\lambda,h,s) \in \R^{4n+4}$ is a non-degenerate zero of the cusp map $F$ defined in \eqref{eq:cusp_map} and that exactly $n-1$ eigenvalues of $D_x f(x,\lambda)$ have nonzero real parts. Then
\begin{equation} \label{eq:nondegeneracy_coeff}
c \bydef \frac{1}{6} w^T \left( D_{xxx} f (x,\lambda)(v,v,v) + 3 D_{xx} f (x,\lambda)(v,h) \right) \ne 0,
\end{equation}
and therefore there is a cusp bifurcation occurs at the point $(x,\lambda)$ and the normal form on the center manifold of $(x,\lambda)$ is given by 
\begin{equation} \label{eq:cusp_normal_form}
\dot y = c y^3 + O(y^4).
\end{equation}
\end{theorem}

\begin{proof}
Assume that $X=(x,v,w,\lambda,h,s) \in \R^{4n+4}$ satisfies $F(X)=0$ and that $DF(X) \in M_{4n+4}(\R)$ is invertible. First of all, since $F(X)=0$, we get from \eqref{eq:cusp_map} that $w^T D_xf(x,\lambda)=0$, $D_xf(x,\lambda)v + s_1 v=0$ and $w^T v=1$. Hence, $w^T D_xf(x,\lambda)v + s_1 w^T v = 0$, that is $s_1=0$. This implies that $D_xf(x,\lambda)v=0$. Recalling \eqref{eq:coefficient_b}, we conclude that the five first requirements for a cusp bifurcation hold, namely
\begin{equation} \label{eq0:proof_nondegeneracy}
f(x,\lambda)=0,~D_xf(x,\lambda)v=0, ~w^T D_xf(x,\lambda)=0, ~w^T D_{xx} f(x,\lambda)(v,v)=0 ~ \text{and}~ w^T v=1.
\end{equation}
By assumption, exactly $n-1$ eigenvalues of $D_x f(x,\lambda)$ have nonzero real parts, so the sixth requirement for a cusp bifurcation holds. It remains to show that $c$ defined in \eqref{eq:nondegeneracy_coeff} is non vanishing. First, since $D_{x} f (x,\lambda) h + s_2 v + D_{xx} f (x,\lambda)(v,v)=0$, then $w^T D_{x} f (x,\lambda) h + s_2 w^T v + w^TD_{xx} f (x,\lambda)(v,v)=0+s_2+0=0$ (where we used $w^T D_{x} f (x,\lambda)=0$ and $w^TD_{xx} f (x,\lambda)(v,v)=0$), and therefore $s_2=0$. We obtain that 
\begin{equation} \label{eq1:proof_nondegeneracy}
D_{x} f (x,\lambda) h + D_{xx} f (x,\lambda)(v,v)=0.
\end{equation}
Therefore $(h,s)=(h,0)$ solves \eqref{eq:bordering_system} and hence $c$ defined in \eqref{eq:nondegeneracy_coeff} coincides with the one of equation \eqref{eq:nondegeneracy_coeff_intro}.  To demonstrate that $c \ne 0$, we argue by contradiction, that is, assume that 
\begin{equation} \label{eq2:proof_nondegeneracy}
w^T \left( D_{xxx} f (x,\lambda)(v,v,v) + 3 D_{xx} f (x,\lambda)(v,h) \right) = 0.
\end{equation}
The remaining of the proof is to produce a non zero vector $\hat{X} \in \R^{4n+4}$ such that $DF(X)\hat{X}=0$, violating the hypothesis that $DF(X)$ is invertible. This process begins by defining $(u_2,\sigma_2) \in \R^{n+1}$ to be the unique vector solving
\begin{equation} \label{eq:bordering_system2}
\begin{pmatrix}
D_{x} f (x,\lambda)^T & w \\ v^T & 0
\end{pmatrix}
\begin{pmatrix}
u_2 \\ \sigma_2
\end{pmatrix}
= 
\begin{pmatrix}
D_{xx} f (x,\lambda)^T(w,v) \\ -v^T h
\end{pmatrix},
\end{equation}
whose existence (and uniqueness) follows from the invertibility of the matrix in the left-hand side of \eqref{eq:bordering_system2} (e.g. see Lemma 5.3 in \cite{MR2071006}). Now, since 
\[
v^T D_{xx} f (x,\lambda)^T(w,v) = w^T D_{xx} f (x,\lambda)(v,v)=0,
\]
we conclude that $\sigma_2=v^T(D_{x} f (x,\lambda)^T u_2+ \sigma_2 w )=  v^T D_{xx} f (x,\lambda)^T(w,v)=0$, and therefore
\begin{equation} \label{eq3:proof_nondegeneracy}
D_{x} f (x,\lambda)^T u_2= D_{xx} f (x,\lambda)^T(w,v).
\end{equation}
Now, combining \eqref{eq1:proof_nondegeneracy} and \eqref{eq3:proof_nondegeneracy}, we get that
\begin{align*}
w^T D_{xx}f(x,\lambda)(v,h) & = h^T D_{xx}f(x,\lambda)^T(v,w) \\
& = h^T D_{x} f (x,\lambda)^T u_2 \\
& = u_2^T D_{x} f (x,\lambda)h \\
& = -u_2^T D_{xx} f (x,\lambda)(v,v)
\end{align*}
that is 
\begin{equation} \label{eq4:proof_nondegeneracy}
w^T D_{xx}f(x,\lambda)(v,h) = -u_2^T D_{xx} f (x,\lambda)(v,v).
\end{equation}
Let
\begin{equation} \label{eq:gamma:proof_nondegeneracy}
\gamma \bydef v^T h \in \R.
\end{equation}
and define $(u_3,\sigma_3) \in \R^{n+1}$ to be the unique vector solving
\begin{equation} \label{eq:bordering_system3}
\hspace{-.24cm}
\begin{pmatrix}
D_{x} f (x,\lambda) & v \\ w^T & 0
\end{pmatrix}
\begin{pmatrix}
u_3 \\ \sigma_3
\end{pmatrix}
= 
\begin{pmatrix}
D_{xxx} f (x,\lambda)(v,v,v)+3D_{xx} f (x,\lambda)(v,h) - \gamma D_{xx} f (x,\lambda)(v,v) \\ -h^T u_2
\end{pmatrix}.
\end{equation}
Now, since we assume that $c=0$, then \eqref{eq2:proof_nondegeneracy} holds and hence 
\begin{equation} \label{eq4b:proof_nondegeneracy}
w^T \left( D_{xxx} f (x,\lambda)(v,v,v)+3D_{xx} f (x,\lambda)(v,h) - \gamma D_{xx} f (x,\lambda)(v,v) \right) =0
\end{equation}
which implies (using a similar argument as before) that $\sigma_3=0$. This implies that 
\begin{equation} \label{eq5:proof_nondegeneracy}
D_{x} f (x,\lambda) u_3 = D_{xxx} f (x,\lambda)(v,v,v)+3D_{xx} f (x,\lambda)(v,h) - \gamma D_{xx} f (x,\lambda)(v,v).
\end{equation}
Denote 
\begin{equation} \label{eq6:proof_nondegeneracy}
\hat{X} \bydef (-v,-h+\gamma v ,u_2,u_3,0,0) \in \R^{4n+4}.
\end{equation}
We conclude the proof by showing that each component of the following vector vanishes
\[
D_X F (X) \hat{X} = 
{\tiny
\begin{pmatrix}
-D_xf(x,\lambda) v 
\\
-[D_{xx}f(x,\lambda)(v,v) + D_xf(x,\lambda)h] + \gamma D_xf(x,\lambda) v
\\
-D_{xx}f(x,\lambda)^T (w,v) + D_xf(x,\lambda)^T u_2
\\
2(-v^Th+\gamma v^Tv)
\\
-w^T h+\gamma w^T v +  v^T u_2
\\
-D_{xxx} f(x,\lambda) (v,v,v) - 3 w^T D_{xx} f(x,\lambda) (v,h) + \gamma D_{xx} f(x,\lambda)(v,v) + D_{x} f(x,\lambda)u_3
\\
-w^T D_{xxx} f(x,\lambda) (v,v,v) - 2 w^T D_{xx} f(x,\lambda) (v,h) + 2 \gamma w^T D_{xx} f(x,\lambda) (v,v) + u_2^T D_{xx} f(x,\lambda)(v,v) 
\\
h^T u_2 + w^T u_3
\end{pmatrix}}. 
\]
The first component vanishes since by construction of $X$, $F(X)=0$ and then $D_xf(x,\lambda) v=0$. The second one equals zero by \eqref{eq1:proof_nondegeneracy}. The construction of $u_2$ in \eqref{eq:bordering_system2} and \eqref{eq3:proof_nondegeneracy} imply that the third component of $D_X F (X) \hat{X}$ vanishes. From the definition of $\gamma$ in \eqref{eq:gamma:proof_nondegeneracy} and using $v^Tv=1$, the fourth component is zero. The fifth one is zero since $w^T h=0$, $w^T v=1$ and $v^T u_2=-v^TH = \gamma$, which follows from the second component of \eqref{eq:bordering_system2}. The sixth one vanishing follows directly from the choice of $u_3$ which solves \eqref{eq5:proof_nondegeneracy}. Combining \eqref{eq2:proof_nondegeneracy}, \eqref{eq4:proof_nondegeneracy} and the fact that $w^T D_{xx} f (x,\lambda)(v,v)=0$ implies that the seventh component vanishes. We use the second component of \eqref{eq:bordering_system3} to obtain that the eighth component of $D_X F (X) \hat{X}$ is zero. This allows to conclude that $D_X F (X)$ has a nontrivial kernel, which contradicts the hypothesis that $X$ is a non-degenerate zero. This concludes the proof that $c \ne 0$. Hence, there is a cusp bifurcation at $(x,\lambda)$ and using the theory from \cite{MR2071006}, the normal form on the center manifold of $(x,\lambda)$ is given by \eqref{eq:cusp_normal_form}.
\end{proof}

\subsection{Newton-Kantorovich Theorem and non-degeneracy of zeros} \label{sec:NKT}

The following theorem (see \cite{MR4336016,MR3542954,MR1639986} for similar statements) provides a constructive method for demonstrating the existence of zeros of nonlinear differentiable mappings. A direct consequence of the result is that solutions so obtained are non-degenerate. In the result below, choose any norm $\| \cdot \|$ on $\R^N$.

\begin{theorem}[\bf A Newton-Kantorovich theorem] \label{thm:NK}
Consider a $\cont^1$ map $F:U \to \R^N$, where $U \subset \R^N$ is open. Fix $r_*>0$ and assume that the two upper bounds $Y,Z \ge 0$ satisfy
\begin{equation} \label{eq:boundsY_Z}
\| A F(\bX)\| \leq Y \quad \text{and} \quad
\sup_{\xi \in \overline{B_{r_*}(\bX)}} \| I - A D_XF(\xi) \| \leq Z.
\end{equation}
If 
\begin{equation} \label{eq:radii_polynomial_inequality}
Z<1 \quad \text{and} \quad r_0 \bydef \frac{Y}{1-Z} < r_*,
\end{equation} 
then for each $r \in \left(r_0, r_* \right]$, there is a unique $\tX \in B_{r}(\bX) $ such that $F(\tX) = 0$. Moreover, the vector $\tX \in \R^N$ is a non-degenerate zero of $F$, that is $D_XF(\tX)$ is invertible. 
\end{theorem}

\begin{proof}
Define the Newton-like operator $T:U \to \R^N$ by $T(X) = X - A F(X)$ and note that $D_XT(X) = I - A  D_XF(X)$. Since \eqref{eq:radii_polynomial_inequality} holds, pick any $r \in \left(r_0, r_* \right]$. Note that $r_0=\frac{Y}{1-Z} < r$ implies that $Zr + Y < r$. The idea of the proof is to show that $T: \overline{B_{r}(\bX)} \to B_{r}(\bX)$ is a contraction. 

For $X_1,X_2 \in \overline{B_{r}(\bX)}$ we use the Mean Value Inequality to get that
\begin{align*}
\|T(X_1) - T(X_2)\| &\le \sup_{\xi \in \overline{B_{r}(\bX)}} \|D_XT(\xi)\| \|X_1 - X_2 \| 
\\
& = \sup_{\xi \in \overline{B_{r}(\bX)}} \|I - A D_XF(\xi)\| \|X_1 - X_2 \| 
\\
&\le \sup_{\xi \in \overline{B_{r_*}(\bX)}} \| I - AD_XF(\xi)\|  \|X_1 - X_2 \| 
\\
& \le Z \|X_1 - X_2 \| .
\end{align*}
Since $ Z < 1 $, $T$ is a contraction on $\overline{B_{r}(\bX)}$. To see that $T$ maps the closed ball into itself (in fact in the open ball) choose $X \in \overline{B_{r}(\bX)}$, and observe that
\begin{align*}
\|T(X) - \bX \| &\leq \| T(X) - T(\bX)\| + \|T(\bX) - \bX \| \\
&\leq Z \|X - \bX\| + \| A F(\bX)\| \\
&\leq Zr + Y < r,
\end{align*}
which shows that $T(X) \in B_{r}(\bX)$ for all $X \in \overline{B_{r}(\bX)}$. It follows from the contraction mapping theorem that there exists a unique $\tX \in \overline{B_{r}(\bX)}$ such that $T(\tX)=\tX \in B_{r}(\bX)$. Since $Z<1$, we get 
\[
 \| I - A D_XF(\bX) \| \le \sup_{\xi \in \overline{B_{r_*}(\bX)}} \| I - A D_XF(\xi) \| \leq Z <1,
\]
and hence $A DF(\bX)$ is invertible using a Neumann series argument. From this we get that $A$ is invertible. By invertibility of $A$ and by definition of $T$, the fixed points of $T$ are in one-to-one correspondence with the zeros of $F$. We conclude that there is a unique $\tX \in B_{r}(\bX)$ such that $F(\tX)=0$. Finally, since 
\[
 \| I - A D_XF(\tX) \| \le \sup_{\xi \in \overline{B_{r_*}(\bX)}} \| I - A D_XF(\xi) \| \leq Z <1,
\]
we get that $A D_XF(\tX)$ is invertible and therefore $D_XF(\tX)$ is invertible. This concludes the proof that $\tX \in \R^N$ is a non-degenerate zero of $F$.
\end{proof}

\begin{remark}
Typically, in Theorem~\ref{thm:NK}, the vector $\bX$ is an approximation $\bX$ of $F=0$ obtained numerically (e.g. using Newton's method) and $A \approx D_XF^{-1}(\bX)$ is an approximate inverse of $D_XF(\bX)$.
\end{remark}

\subsection{Rigorous enclosure of \boldmath$\sigma(D_x f(x,\lambda))$\unboldmath} \label{sec:spectrum_enclosure}

In this short section, we demonstrate how to control the spectrum $\sigma(D_x f(x,\lambda))$, assuming that the matrix $D_x f(x,\lambda)$ is diagonalizable over the complex numbers, which is a generic condition. The idea is to introduce a pseudo-diagonalization technique and then use the Gershgorin Circle Theorem to control all eigenvalues. First, assume that using Theorem~\ref{thm:NK}, we proved the existence of $X=(x,v,w,\lambda,h,s) \in B_{r}(\bX) $ such that $F(X) = 0$ for some small $r>0$. Denote $\bX=(\bar x,\bar v,\bar w,\bar \lambda,\bar h,0) \in \R^{4n+4}$. Note that $D_x f(x,\lambda)$ is a matrix whose entries can be rigorously enclosed using the control that we have over the location of $(x,\lambda)$. For instance, if the norm in $\R^{4n+4}$ is the sup norm $\| \cdot \| = \| \cdot \|_\infty$, then we know that 
\begin{equation} \label{eq:control_of_entries}
|\lambda_j - \bar \lambda_j| \le r, \text{ for } j=1,2 \quad \text{and} \quad
|x_i - \bar x_i| \le r, \text{ for } i=1,\dots,n.
\end{equation}
Using interval arithmetic and the control from \eqref{eq:control_of_entries}, we can construct an interval matrix $\bf M$ such that $D_x f(x,\lambda) \subset {\bf M}$, where the inclusion is to be understood component-wise. Then, using a standard numerical linear algebra techniques, compute a list of approximate numerical eigenvectors $\bar \xi_1,\dots,\bar \xi_n \in \mathbb{C}^n$ and define the matrix $P$ columnwise as follows: $P=[\bar \xi_1,  \bar \xi_2, \ldots, \bar\xi_n] \in M_n(\mathbb{C})$. Consider the matrix $\Lambda \bydef P^{-1} D_x f(x,\lambda) P$, where $P^{-1}$ is the exact inverse of the matrix $P$, and hence $\sigma(D_x f(x,\lambda))=\sigma(\Lambda)$. Since the vectors $\bar \xi_1,\dots,\bar \xi_n$ are only approximate eigenvectors, we call this process a pseudo-diagonalization. The idea now is to apply Gershgorin Circle Theorem to the matrix $\Lambda$ using the control $D_x f(x,\lambda) \subset {\bf M}$, that is using that 
\[
\Lambda \subset {\bf \Lambda} \bydef P^{-1} {\bf M} P,
\]
where ${\bf \Lambda}$ is an interval matrix which is almost diagonal. Denote $\Lambda= (\Lambda_{i,j})_{i,j=1}^n$. For each $i=1,\dots,n$, denote the Gershgorin disks by
\[
D_i=\{ z \in \C : |z-\Lambda_{i,i}| \le \sum_{j=1 \atop j \ne i}^n |\Lambda_{i,j}| \}.
\]
While the exact computation of $D_i$ is in general out of reach, we can resort to interval arithmetic and use that $\Lambda \subset {\bf \Lambda}$ to construct explicit sets $\hat D_i$ such that $D_i \subset \hat D_i$. Hence, if we can show that the sets $\hat D_i$ are mutually disjoint, we conclude from Gershgorin Circle Theorem (e.g. see \cite{MR832183}) that
\[
\sigma(D_x f(x,\lambda)) \subset \bigcup_{i=1}^n D_i \subset \bigcup_{i=1}^n \hat D_i
\]
where $\hat D_i$ contains exactly one eigenvalue of $D_x f(x,\lambda)$.

\subsection{Procedure for proving constructively the existence of cusp bifurcations} \label{sec:procedure_cusp}

In practice, the approach to prove the existence of a cusp bifurcation is as follows. Combining the two-parameter continuation and Poincaré index theory method presented in Section~\ref{sec:detection}, we detect the presence of a cusp. Then, using the approximate location of the cusp, we apply Newton's method on the cusp map $F$ defined in \eqref{eq:cusp_map} to obtain a refined numerical approximation $\bX$ for a solution of $F=0$. Afterwards, we apply Theorem~\ref{thm:NK} to prove that the cusp map defined in \eqref{eq:cusp_map} has a non-degenerate zero $X=(x,v,w,\lambda,h,s)$ in the vicinity of this numerical approximation. Using the approach of Section~\ref{sec:spectrum_enclosure}, we verify that exactly $n-1$ eigenvalues of $D_x f(x,\lambda)$ have nonzero real parts using the method of \ref{sec:spectrum_enclosure}. We conclude using Theorem~\ref{thm:cusp} that there is a cusp bifurcation at $(x,\lambda)$. Using the explicit control for the location $(x,\lambda)$, we may finally compute a rigorous enclosure for the normal form coefficient $c$ defined in \eqref{eq:nondegeneracy_coeff}.

\section{Applications} \label{sec:applications}
%!TEX root = cusp_paper_revised.tex

In this section, we apply the procedure presented in this paper (as laid out explicitly in Section~\ref{sec:procedure_cusp}) to detect and prove existence of cusp bifurcations in four model problems. In all examples, we choose the sup norm $\| \cdot \|_\infty$ to perform the computer-assisted proofs. Although other norms could be used in principle, the choice of the sup norm is motivated by two main factors. First, its unit ball is a hypercube aligned with the coordinate axes, making it simpler to describe and apply in various contexts. Second, when verifying the assumptions of Theorem~\ref{thm:NK}, calculating the bound \( Z \) in \eqref{eq:boundsY_Z} involves computing a matrix norm. This is straightforward with the sup norm but becomes significantly more computationally intensive with other norms, such as the Euclidean norm, which requires finding the largest singular value.

\subsection{Bazykin model} \label{sec:bazykin}

Consider the two-dimensional model
\begin{equation} \label{eq:bazykin_eq}
\dot x = f(x,\lambda) = \begin{pmatrix}
\displaystyle
x_1 - \frac{x_1 x_2}{1+\lambda_1 x_1} - 0.01 x_1^2 
\vspace{.1cm}
\\
\displaystyle
- x_2 +  \frac{x_1 x_2}{1+\lambda_1 x_1} - \lambda_2 x_2^2
\end{pmatrix},
\end{equation}
which was consider in \cite{MR801544,MR2071006} to model the dynamics of a predator-prey ecosystem, where the variable $x_1$ indicates the (scaled) prey population and the variable $x_2$ indicates the (scaled) predator population. The nonnegative parameters $\lambda_1$ and $\lambda_2$ describe the behaviour of isolated populations and their interaction. The recent approach proposed in \cite{OT_cusp} introduces a constrained optimization problem whose optima correspond to cusp bifurcations, and the optima are computed with a Lagrange multipliers method. 

We first computed a triangulation of the portion of the manifold of equilibria for \eqref{eq:bazykin_eq} contained inside the region $\{(x_1,x_2)\in[0,35]\times[0,35], (\lambda_1,\lambda_2)\in[0,4]\times[0,4]\}$, starting from the explicitly computed equilibrium $(x_1,x_2,\lambda_1,\lambda_2)=\left[10,10,\frac{91}{90},-\frac{1}{100}\right]$. The computation resulted in a triangulation made of approximately 2700 nodes and 5300 simplices. See Figure \ref{fig:example1}. The detection procedure flagged 2 triangles as \verb"CUSP" and all the rest as \verb"NO CUSP"; notably, no triangles were flagged as \verb"UNDETERMINED".

Finally, the centers of the two \verb"CUSP" triangles were handled by the computer-assisted proof presented in Section \ref{sec:CAPs}. This leads to the following two numerical approximations for the location of cusps, namely $\left(\bar x^{(1)},\bar \lambda^{(1)}\right)$ and $\left(\bar x^{(2)},\bar \lambda^{(2)}\right)$ given by
\begin{align*}
\bar x^{(1)} &= 
\begin{pmatrix} 
25.823060900508402 \\
2.442085531687371
\end{pmatrix}, 
\quad 
\bar \lambda^{(1)} = \begin{pmatrix} 0.088767308061008 \\ 2.802361210268358 \end{pmatrix}
\\
\bar x^{(2)} &= \begin{pmatrix} 
  32.593605766158269 \\
  20.474303357201517
\end{pmatrix}, \quad 
\bar \lambda^{(2)} 
= \begin{pmatrix} 0.901232691938992 \\ 0.003568419361272\end{pmatrix}.
\end{align*}
Then, we fixed $r_* = 10^{-12}$ and demonstrated that the exact cusps bifurcations occurred at $\left(\tilde x^{(1)},\tilde \lambda^{(1)}\right)$ and $\left(\tilde x^{(2)},\tilde \lambda^{(2)}\right)$, where $\|\tilde x^{(i)}-\bar x^{(i)}\|_\infty, \|\tilde \lambda^{(i)}-\bar \lambda^{(i)}\|_\infty \le r^{(i)}$, where $r^{(1)} = 8 \times 10^{-14}$ and $r^{(2)} = 4 \times 10^{-14}$. As a comparison, the values obtained in \cite{OT_cusp} are
\begin{equation} \label{eq:numerics_OT_cusp}
\lambda^{(1)} = (0.0887673081, 2.\underline{7983226350}) \quad \text{and} \quad \lambda^{(2)} = (0.9012326919, 0.00356\underline{31988})
\end{equation}
and hence our method shows that the underlined digits from the numerical results of \cite{OT_cusp} presented in \eqref{eq:numerics_OT_cusp} are inaccurate. 

Using the rigorous control over the location of the exact solutions of the cusp map, we proved that the coefficients of the normal form $\dot y = c y^3 + O(y^4)$ on the center manifold at each cusp satisfy $c^{(i)} \in \bar c^{(i)} + [-1,1]r_c^{(i)}$, where 
\begin{align*}
\bar c^{(1)} &= -1.948784809507697 \times 10^{-4}, \quad r_c^{(1)} = 1.1 \times 10^{-14} \\
\bar c^{(2)} &= 7.590018095668060 \times 10^{-5}, \quad r_c^{(2)} = 2.1 \times 10^{-15}.
\end{align*}
Hence, $c^{(1)}<0$ and the dynamics on the center manifold of the first cusp $\left(\tilde x^{(1)},\tilde \lambda^{(1)}\right)$ is stable. Moreover, $c^{(2)}>0$ and the dynamics on the center manifold of the second cusp $\left(\tilde x^{(2)},\tilde \lambda^{(2)})\right)$ is unstable.

\begin{figure}[h]
\centering
\includegraphics[width=.8\textwidth]{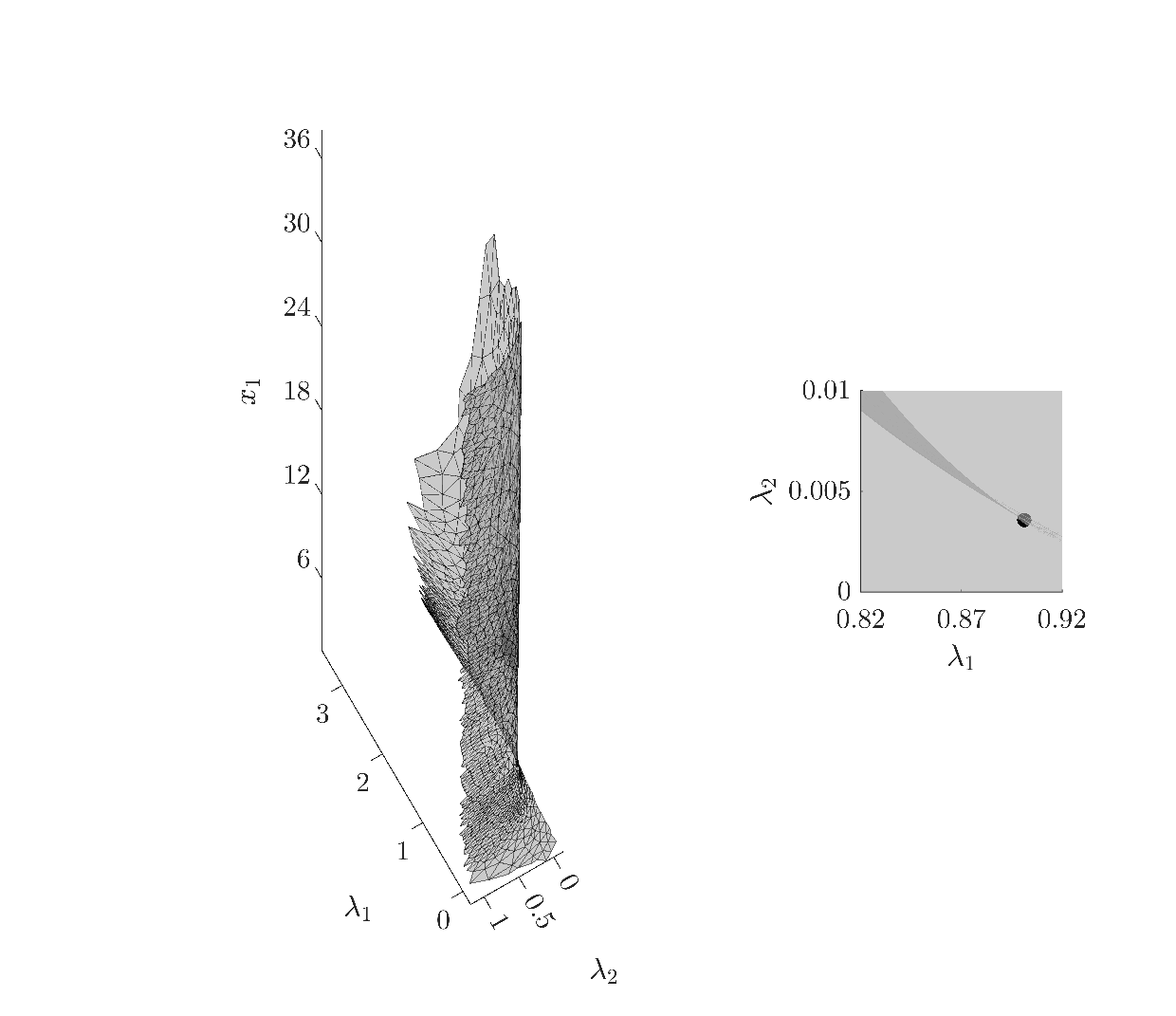}
\caption{On the left the projection of the triangulation for the manifold of equilibria of \eqref{eq:bazykin_eq} on the $(\lambda_1,\lambda_2,x_1)$-space; on the right a close up of the projection on the $(\lambda_1,\lambda_2)$-plane near the second cusp point (black circular mark). The darkest shade of grey identifies the region where three equilibria exist for the same value of the parameter.}\label{fig:example1}
\end{figure}

\subsection{Predator-prey systems with non-monotonic response function} \label{sec:predator-prey}

In \cite{MR1951954}, the following model 
\begin{equation} \label{eq:Predator-prey_eq}
\dot x =
f(x,\lambda) \bydef 
\begin{pmatrix}
\displaystyle
x_1 \left( 1 - \kappa x_1 - \frac{x_2}{\lambda_1 x_1^2 + \lambda_2 x_1 + 1} \right)
\\
\displaystyle
x_2 \left( -\delta - \mu x_2 + \frac{x_1}{\lambda_1 x_1^2 + \lambda_2 x_1 + 1} \right)
\end{pmatrix}
%= 
%\begin{pmatrix}
%x_1 - \kappa x_1^2 - h(x,\lambda)
%\\
%-\delta x_2 - \mu x_2^2 + h(x,\lambda)
%\end{pmatrix}, \qquad h(x,\lambda) \bydef \frac{x_1x_2}{\lambda_1 x_1^2 + \lambda_2 x_1 + 1},
\end{equation}
was proposed as a predator-prey system with a nonmonotonic functional response $y \mapsto \frac{my}{ay^2+by+1}$, where it is assumed that the rate of conversion of captured prey is proportional to this function. The variables $x_1$ and $x_2$ are the prey and predator densities, respectively, while the parameter $\delta$ measures the mortality rate of the predator, $\kappa$ and $\mu$ measure the intraspecific competition. The predation factor is measured by parameters $\lambda_1$ and $\lambda_2$. As in \cite{HARJANTO2016188,OT_cusp,MR4043243}, we fix the parameters $\delta= 1.1$, $\kappa=0.01$, $\mu= 0.1$, and leave the parameters $\lambda_1$ and $\lambda_2$ free. 

We computed a triangulation of the portion of the manifold of equilibria for \eqref{eq:Predator-prey_eq} contained inside the region $\{(x_1,x_2)\in[0,20]\times[0,20], (\lambda_1,\lambda_2)\in[-3,3]\times[-3,3]\}$, starting from the explicitly computed equilibrium $(x_1,x_2,\lambda_1,\lambda_2)=\left[\frac{1200}{1001},\frac{989}{1001},0,0\right]$. The computation resulted in a triangulation made of approximately 1300 nodes and 2500 simplices. See Figure \ref{fig:example2}. The detection procedure flagged 3 triangles as \verb"CUSP" and 8 triangles as \verb"UNDETERMINED" on step 6 of the procedure. We performed the proof on all these triangles and only one of them turned out to contain a cusp point. Further inspection revealed that the triangles that yielded a \verb"CUSP" or \verb"UNDETERMINED" flag were very large, a fact that happens on portions of the manifold that are locally very flat (we recall that the size of triangles is inversely proportional to local curvature). This is a warning that tolerances for the triangulation procedure should be set so to limit the size of triangles. After tightening the tolerances to reduce triangle sizes and reapplying the method, only one triangle was flagged as \texttt{CUSP}, and none were flagged as \texttt{UNDETERMINED}.

For the actual cusp point, we fixed $r_* = 5 \times 10^{-12}$ and applied the approach Section \ref{sec:CAPs} to \eqref{eq:Predator-prey_eq}, obtaining the following numerical approximation for the cusp 
\[
\bar x = \begin{pmatrix} 
  13.196245980692725\\
   6.533219395017792
\end{pmatrix}, \quad 
\bar \lambda
= \begin{pmatrix} -0.004206696164016 \\ 0.550079382002905 \end{pmatrix}.
\]
and proved that there exists a cusp bifurcation at $(\tx,\tilde \lambda)$ where $\|\tx-\bar x\|_\infty,\|\tilde \lambda -\bar \lambda \|_\infty < r = 1.42 \times 10^{-12}$. As a comparison, the values obtained in \cite{OT_cusp} are
\begin{equation} \label{eq:numerics_OT_cusp2}
\lambda = (-0.00420\underline{0565504},0.55007938\underline{196})
\end{equation}
and hence our method shows that the underlined digits from the numerical results of \cite{OT_cusp} presented in \eqref{eq:numerics_OT_cusp2} are inaccurate. 

Using the rigorous control over the location of the exact solution of the cusp map, we proved that the coefficient of the normal form $\dot y = c y^3 + O(y^4)$ on the center manifold at each cusp satisfy $c \in \bar c + [-1,1]r_c$, where 
\[
\bar c = -0.005350994279567, \quad r_c = 2.6 \times 10^{-11}.
\]
Hence, $c<0$ and the dynamics on the center manifold of the cusp $(\tilde x,\tilde \lambda)$ is stable. 

\begin{figure}[h]
\centering
\includegraphics[width=.9\textwidth]{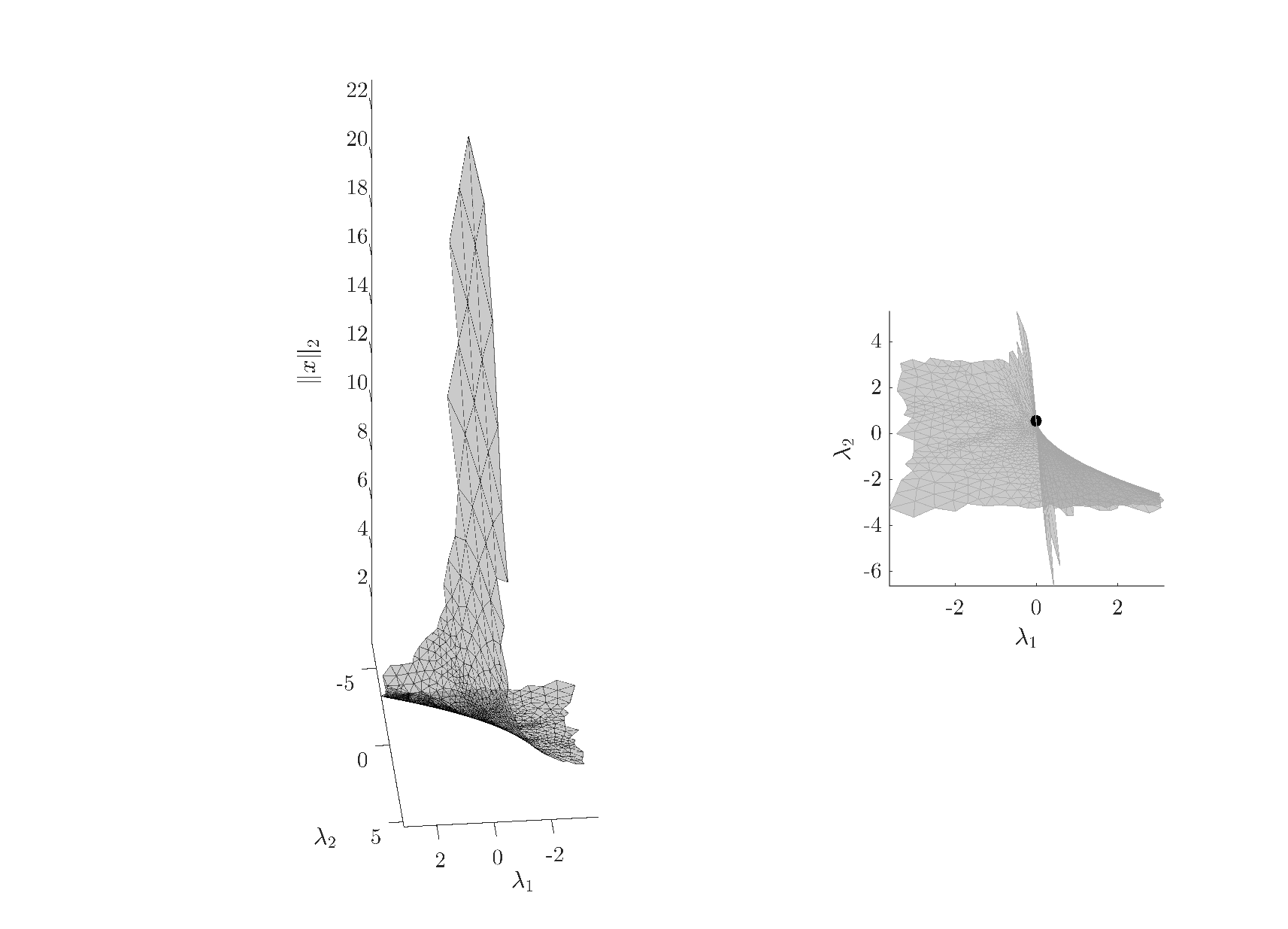}
\caption{On the left the projection of the triangulation of the manifold of equilibria of \eqref{eq:Predator-prey_eq} on the $(\lambda_1,\lambda_2,\Vert x\Vert_2)$-space; on the right the projection on the $(\lambda_1,\lambda_2)$-plane; a black circular mark signals the location of the cusp point.}\label{fig:example2}
\end{figure}

\subsection{Bykov model} \label{sec:bykov}

The following chemical model, introduced in \cite{Bykov}, describes carbon monoxide (CO) catalytic reaction (oxidation) on platinum
\begin{equation} \label{eq:bykov_eq}
\dot x =
f(x,\lambda) \bydef 
\begin{pmatrix}
5 (1-x_1-x_2-x_3)^2  - 2 x_1^2 - 10 x_1 x_2
\\
\lambda_1 (1-x_1-x_2-x_3) - 0.1 x_2 - 10 x_1 x_2
\\
0.0675 (1-x_1-x_2-x_3 - \lambda_2 x_3)
\end{pmatrix}.
\end{equation}
This system is often called the Bykov-Yablonksi-Kim model and has been studied numerically in \cite{MR1803181,kuznetsov_tutorial}. In particular, a cusp bifurcation has been numerically observed in \cite{kuznetsov_tutorial}. %We numerically found six candidate cusp bifurcations, only one of which is physically meaningful (that is with positive $x_i$ and $\lambda_i$). 

We investigated the region $\{\min([x_1,x_2,x_3,\lambda_1,\lambda_2])\ge0, \max([x_1,x_2,x_3,\lambda_1,\lambda_2])\le1.6 \}$. Starting from an equilibrium explicitly obtained by setting $\lambda_1=\lambda_2=1$ and solving with respect to $x_1,x_2,x_3$, we obtained a triangulation made of approximately 9300 nodes and 18500 simplices. See Figure \ref{fig:example3}. Only one triangle was flagged as \verb"CUSP", while all the rest were flagged as \verb"NO CUSP". 

The procedure in Section \ref{sec:CAPs} was then used to find the following numerical approximation of the cusp:
\[
\bar x = 
\begin{pmatrix} 
   0.035940636548892 \\
   0.352005430946603 \\
   0.451370111738407
\end{pmatrix}
\quad \text{and} \quad
\bar \lambda =
\begin{pmatrix}
1.006408329678319 \\
0.355991273208678
\end{pmatrix}.
\]
We fixed $r_* = 10^{-10}$ and proved that there exists a cusp bifurcation at $(\tx,\tilde \lambda)$ where $\|\tx-\bar x\|_\infty,\|\tilde \lambda -\bar \lambda \|_\infty < r = 3.7 \times 10^{-13}$. Using the rigorous control over the location of the exact solution of the cusp map, we proved that the coefficient of the normal form $\dot y = c y^3 + O(y^4)$ on the center manifold at each cusp satisfy $c \in \bar c + [-1,1]r_c$, where 
\[
\bar c = 0.362788656889452, \quad r_c = 1.7 \times 10^{-11}.
\]
Hence, $c>0$ and the dynamics on the center manifold of the cusp $(\tilde x,\tilde \lambda)$ is unstable. 

\begin{figure}[h]
\centering
\includegraphics[width=1\textwidth]{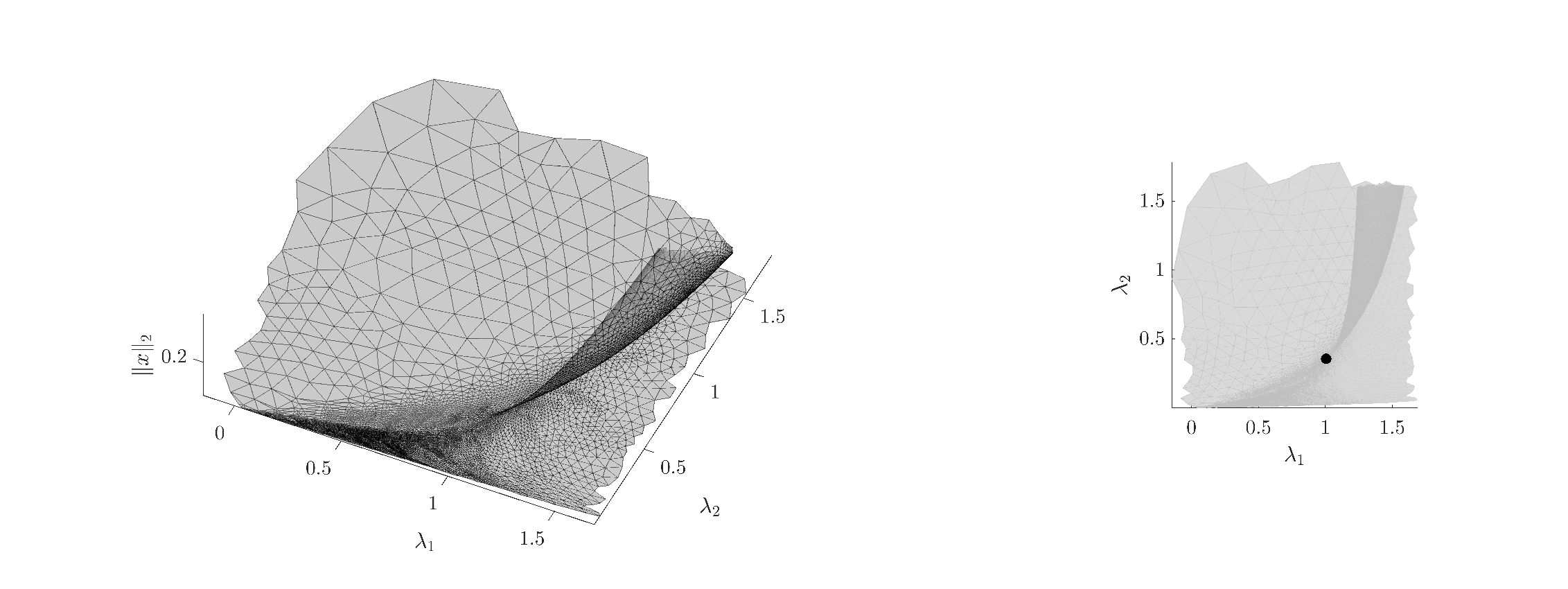}
\caption{On the left the projection of the triangulation of the manifold of equilibria of \eqref{eq:bykov_eq} on the $(\lambda_1,\lambda_2,\Vert x\Vert_2)$-space; on the right the projection on the $(\lambda_1,\lambda_2)$-plane; a black circular mark signals the location of the cusp point.}\label{fig:example3}
\end{figure}

\subsection{Metastatic cell transition as an ODE model} \label{sec:metastatic}

Consider the model introduced in \cite{pnas_cancer} 
\begin{equation} \label{eq:shub_eq}
\dot x =
f(x,\lambda) \bydef 
\begin{pmatrix}
\displaystyle
\frac{1}{1+x_3} - \lambda_1 x_1
\vspace{.15cm}
\\
\displaystyle
1000 \frac{x_1^5}{32+x_1^5} - x_2 - 200x_2 x_3
\vspace{.15cm}
\\
\displaystyle
0.02 +  19.98 \frac{\lambda_2^3}{\lambda_2^3+x_3^3} - x_3 - 200 x_2 x_3
\end{pmatrix}
\end{equation}
where $x_1$, $x_2$ and $x_3$ represent the proteins {\tt RKIP}, {\tt let-7} and {\tt BACH1} which interact in the cell and are relevant for determining breast cancer metastasis. The parameters $\lambda_1$ and $\lambda_2$ describe the instability of {\tt RKIP} and insensitivity of {\tt BACH1} to self-regulation, respectively. The existence of cusp bifurcations was numerically investigated by Delamonica et al. in \cite{shub}. %We numerically found the existence of five cusp candidates, but only one had positive entries for the protein values and the parameters, namely

We computed a triangulation of the portion of manifold of equilibria for \eqref{eq:shub_eq} inside the region $\{(x_1,x_2,x_3)\in [0,3]\times[0,3]\times[0,3], (\lambda_1,\lambda_2)\in [.05, 1.5]\times [.05, 1.5]\}$ starting from an equilibrium explicitly computed by imposing $\lambda_1=\lambda_2=1$. The triangulation had approximately 2600 nodes and 4900 simplices. See Figure \ref{fig:example4}. All simplices except one were flagged as \verb"NO CUSP". One was flagged as \verb"CUSP" and led to the following numerical approximation:
\begin{align*}
\bar x &= 
\begin{pmatrix} 
   0.932145719912463\\
   2.218402699305749\\
   0.043501045476190
\end{pmatrix}
\quad \text{and} \quad
\bar \lambda =
\begin{pmatrix} 
   1.028071456932951\\
   0.134352688793418
\end{pmatrix}.
%\\
%\bar x^{(2)} &= 
%\begin{pmatrix} 
%   0.557008178438887\\
%   0.000374844468130\\
%  22.307575624236815
%\end{pmatrix}
%\quad \text{and} \quad
%\bar \lambda^{(2)} =
%\begin{pmatrix} 
%   0.077026706032790\\
% -40.580912449543433
%\end{pmatrix}.
\end{align*}

We fixed $r_* = 10^{-12}$ and demonstrated that the exact cusp bifurcation occurs at $(\tilde x,\tilde \lambda)$, where $\|\tilde x-\bar x\|_\infty, \|\tilde \lambda-\bar \lambda\|_\infty \le 3.1 \times 10^{-13}$. Using the rigorous control over the location of the exact solution of the cusp map, we proved that the coefficient of the normal form $\dot y = c y^3 + O(y^4)$ on the center manifold satisfies $c \in \bar c + [-1,1]r_c$, where $\bar c= -0.062132151368075$ and $r_c= 3.8 \times 10^{-11}$. Hence, $c<0$ and the dynamics on the center manifold of the cusp $(\tilde x,\tilde \lambda)$ is stable.

\begin{figure}[h]
\centering
\includegraphics[width=1\textwidth]{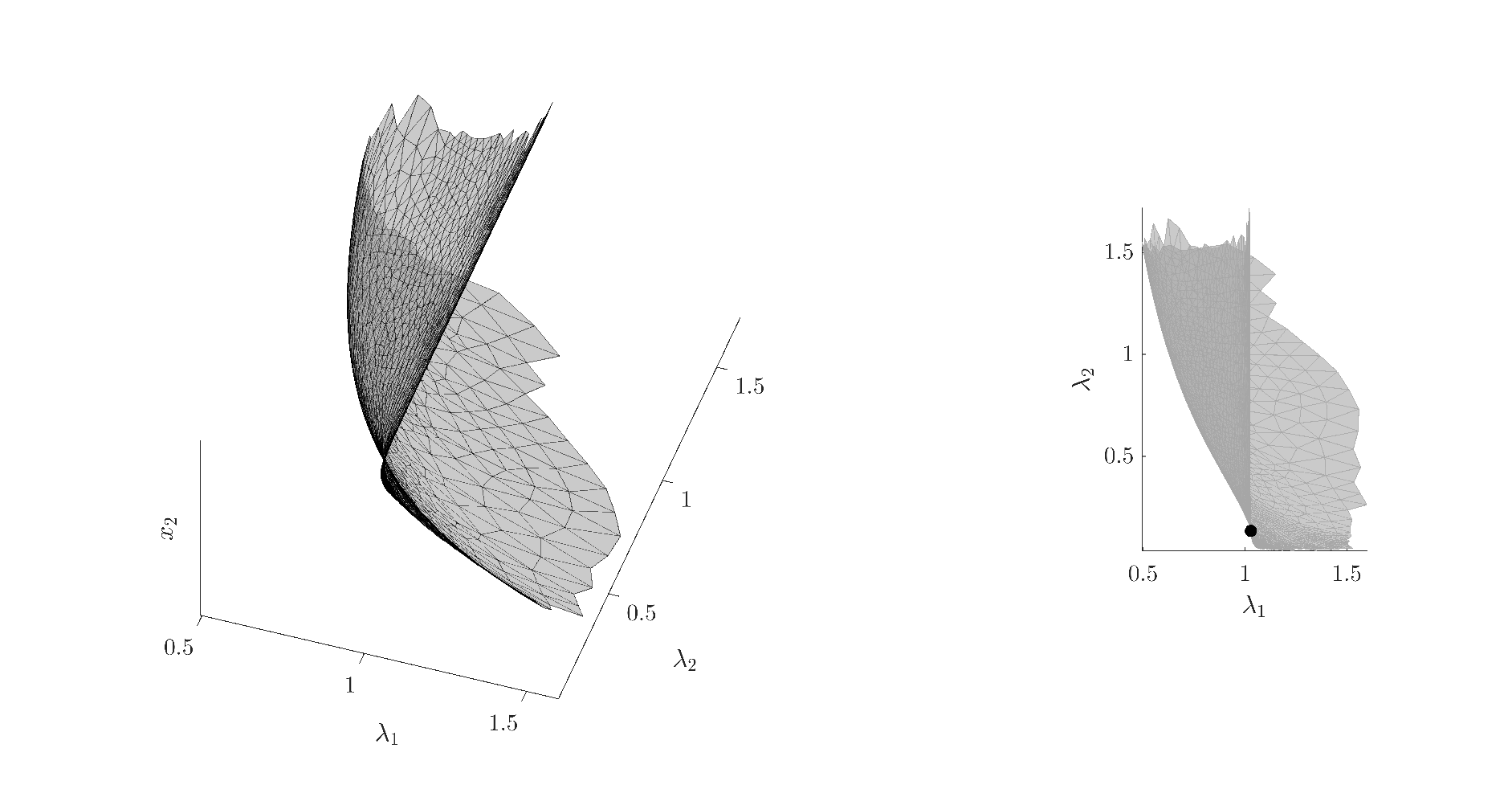}
\caption{On the left the projection of the triangulation of the manifold of equilibria of \eqref{eq:shub_eq} on the $(\lambda_1,\lambda_2,x_2)$-space; on the right the projection on the $(\lambda_1,\lambda_2)$-plane; a black circular mark signals the location of the cusp point.}\label{fig:example4}
\end{figure}

%\begin{example}[\bf Koper model]
%Consider the model introduced and studied by Koper in \cite{MR1307333}
%%
%\begin{equation} \label{eq:koper_eq}
%\dot x =
%f(x,\lambda) \bydef 
%\begin{pmatrix}
%10 \left(\lambda_1 x_2 - x_1^3 + 3x_1 - \lambda_2 \right)
%\\
%x_1 - 2 x_2 + x_3 
%\\
%x_2 - x_3
%\end{pmatrix}
%\end{equation}
%%
%which is a three-dimensional van der Pol-Duffing oscillator, and later studied numerically in \cite{matcont_cusp}. We fixed $k=1$, applied Newton to the cusp map $F:\R^{11} \to \R^{11}$ given in \eqref{eq:cusp_map} and, starting from random initial conditions in $\R^{11}$ converged (almost always) to the cusp point $\tilde X = (\tilde x,\tilde \lambda,\tilde v, \tilde w) \in \R^{11}$, where
%\[
%\tilde x = (0,0,0)^T, \quad 
%\tilde \lambda = (-3,0)^T, \quad 
%\tilde v = \frac{1}{\sqrt{3}}(1,1,1)^T 
%\quad \text{and} \quad 
%\tilde w = \frac{\sqrt{3}}{59}(-1,30,30)^T.
%\]
%%
%The zero $\tilde X$ is numerically non degenerate with $\left\| DF(\tilde X)^{-1} \right\|_\infty \approx 10.57$. The normal form coefficient is given by $c \approx 0.056497$.
%\end{example}

%\bibliographystyle{unsrt}
\bibliographystyle{plain}
\bibliography{papers}

\begin{thebibliography}{10}

\bibitem{MR2679365}
Gianni Arioli and Hans Koch.
\newblock Computer-assisted methods for the study of stationary solutions in
  dissipative systems, applied to the {K}uramoto-{S}ivashinski equation.
\newblock {\em Arch. Ration. Mech. Anal.}, 197(3):1033--1051, 2010.

\bibitem{MR801544}
A.~D. Bazykin.
\newblock {\em Mathematical biophysics of interacting populations}.
\newblock ``Nauka'', Moscow, 1985.

\bibitem{MR1645376}
M.~L. Brodzik.
\newblock The computation of simplicial approximations of implicitly defined
  {$p$}-dimensional manifolds.
\newblock {\em Comput. Math. Appl.}, 36(6):93--113, 1998.

\bibitem{MR1298047}
M.~L. Brodzik and W.~C. Rheinboldt.
\newblock The computation of simplicial approximations of implicitly defined
  two-dimensional manifolds.
\newblock {\em Comput. Math. Appl.}, 28(9):9--21, 1994.

\bibitem{Bykov}
V.I. Bykov, G.S. Yablonskii, and V.F. Kim.
\newblock On the simple model of kinetic self-oscillations in catalytic
  reaction of co oxidation.
\newblock {\em Doklady AN USSR (Chemistry)}, 242(3):637--639, 1978.

\bibitem{MR3397322}
Maciej~J. Capi\'{n}ski and Piotr Zgliczy\'{n}ski.
\newblock Geometric proof for normally hyperbolic invariant manifolds.
\newblock {\em J. Differential Equations}, 259(11):6215--6286, 2015.

\bibitem{MR3567489}
Maciej~J. Capi\'{n}ski and Piotr Zgliczy\'{n}ski.
\newblock Beyond the {M}elnikov method: a computer assisted approach.
\newblock {\em J. Differential Equations}, 262(1):365--417, 2017.

\bibitem{chow1988bifurcations}
Shui-Nee Chow and Yun-Qiu Shen.
\newblock Bifurcations via singular value decompositions.
\newblock {\em Applied Mathematics and Computation}, 28(3):231--245, 1988.

\bibitem{MR4337868}
Kevin E.~M. Church and Jean-Philippe Lessard.
\newblock Rigorous verification of {H}opf bifurcations in functional
  differential equations of mixed type.
\newblock {\em Phys. D}, 429:Paper No. 133072, 13, 2022.

\bibitem{church_queirolo}
Kevin E.~M. Church and Elena Queirolo.
\newblock Computer-assisted proofs of {H}opf bubbles and degenerate {H}opf
  bifurcations.
\newblock {\em Journal of Dynamics and Differential Equations}, 2023.
\newblock arXiv:2202.13326.

\bibitem{MR0069338}
Earl~A. Coddington and Norman Levinson.
\newblock {\em Theory of ordinary differential equations}.
\newblock McGraw-Hill Book Co., Inc., New York-Toronto-London, 1955.

\bibitem{CoDeLoPu}
Alessandro Colombo, Nicoletta~Del Buono, Luciano Lopez, and Alessandro
  Pugliese.
\newblock Computational techniques to locate crossing/sliding regions and their
  sets of attraction in non-smooth dynamical systems.
\newblock {\em Discrete and Continuous Dynamical Systems - B},
  23(7):2911--2934, 2018.

\bibitem{MR4336016}
Kevin Constantineau, Carlos Garc\'ia-Azpeitia, and Jean-Philippe Lessard.
\newblock Spatial relative equilibria and periodic solutions of the {C}oulomb
  {$(n+1)$}-body problem.
\newblock {\em Qual. Theory Dyn. Syst.}, 21(1):Paper No. 3, 19, 2022.

\bibitem{shub}
Brenda Delamonica, G{\'a}bor Bal{\'a}zsi, and Michael Shub.
\newblock Cusp bifurcation in metastatic breast cancer cells.
\newblock arXiv:2203.15888, 2022.

\bibitem{MR2000880}
A.~Dhooge, W.~Govaerts, and Yu.~A. Kuznetsov.
\newblock M{ATCONT}: a {MATLAB} package for numerical bifurcation analysis of
  {ODE}s.
\newblock {\em ACM Trans. Math. Software}, 29(2):141--164, 2003.

\bibitem{Dieci1999OnSD}
Luca Dieci and Timo Eirola.
\newblock On smooth decompositions of matrices.
\newblock {\em SIAM J. Matrix Anal. Appl.}, 20:800--819, 1999.

\bibitem{DIECI2011996}
Luca Dieci, M.~Grazia Gasparo, Alessandra Papini, and Alessandro Pugliese.
\newblock Locating coalescing singular values of large two-parameter matrices.
\newblock {\em Mathematics and Computers in Simulation}, 81(5):996 -- 1005,
  2011.
\newblock Important aspects on structural dynamical systems and their numerical
  computation.

\bibitem{dieci2006path}
Luca Dieci, Maria~Grazia Gasparo, and Alessandra Papini.
\newblock Path following by svd.
\newblock In {\em Computational Science--ICCS 2006}, pages 677--684. Springer,
  2006.

\bibitem{DIECI20081255}
Luca Dieci and Alessandro Pugliese.
\newblock Singular values of two-parameter matrices: an algorithm to accurately
  find their intersections.
\newblock {\em Mathematics and Computers in Simulation}, 79(4):1255--1269,
  2008.

\bibitem{MR2530255}
Luca Dieci and Alessandro Pugliese.
\newblock Two-parameter {SVD}: coalescing singular values and periodicity.
\newblock {\em SIAM J. Matrix Anal. Appl.}, 31(2):375--403, 2009.

\bibitem{MR635945}
Eusebius Doedel.
\newblock A{UTO}: a program for the automatic bifurcation analysis of
  autonomous systems.
\newblock {\em Congr. Numer.}, 30:265--284, 1981.

\bibitem{MR3464215}
Marcio Gameiro, Jean-Philippe Lessard, and Alessandro Pugliese.
\newblock Computation of smooth manifolds via rigorous multi-parameter
  continuation in infinite dimensions.
\newblock {\em Found. Comput. Math.}, 16(2):531--575, 2016.

\bibitem{MR1803181}
W.~Govaerts.
\newblock Numerical bifurcation analysis for {ODE}s.
\newblock {\em J. Comput. Appl. Math.}, 125(1-2):57--68, 2000.
\newblock Numerical analysis 2000, Vol. VI, Ordinary differential equations and
  integral equations.

\bibitem{HARJANTO2016188}
Eric Harjanto and J.M. Tuwankotta.
\newblock Bifurcation of periodic solution in a predator--prey type of systems
  with non-monotonic response function and periodic perturbation.
\newblock {\em International Journal of Non-Linear Mechanics}, 85:188--196,
  2016.

\bibitem{MR4217108}
Wouter Hetebrij and J.~D. Mireles~James.
\newblock Critical homoclinics in a restricted four-body problem: numerical
  continuation and center manifold computations.
\newblock {\em Celestial Mech. Dynam. Astronom.}, 133(2):Paper No. 5, 48, 2021.

\bibitem{hirsch_diffTopology}
M.W. Hirsch.
\newblock {\em Differential Topology}.
\newblock Springer-Verlag, New York, 1976.

\bibitem{MR832183}
Roger~A. Horn and Charles~R. Johnson.
\newblock {\em Matrix analysis}.
\newblock Cambridge University Press, Cambridge, 1985.

\bibitem{Kanzawa_Oishi}
Y.~Kanzawa and S.~Oishi.
\newblock Calculating bifurcation points with guaranteed accuracy.
\newblock {\em IEICE Trans. Fundamentals E82-A 6}, pages 1055--1061, 1999.

\bibitem{MR910499}
H.~B. Keller.
\newblock {\em Lectures on numerical methods in bifurcation problems},
  volume~79 of {\em Tata Institute of Fundamental Research Lectures on
  Mathematics and Physics}.
\newblock Published for the Tata Institute of Fundamental Research, Bombay,
  1987.
\newblock With notes by A. K. Nandakumaran and Mythily Ramaswamy.

\bibitem{shane_elena2023}
Shane Kepley, Konstantin Mischaikow, and Elena Queirolo.
\newblock Global analysis of regulatory network dynamics: equilibria and
  saddle-node bifurcations.
\newblock arXiv:2204.13739, 2023.

\bibitem{MR2351028}
Hiroshi Kokubu, Daniel Wilczak, and Piotr Zgliczy{\'n}ski.
\newblock Rigorous verification of cocoon bifurcations in the {M}ichelson
  system.
\newblock {\em Nonlinearity}, 20(9):2147--2174, 2007.

\bibitem{kuehn_queirolo}
Christian Kuehn and Elena Queirolo.
\newblock Computer validation of neural network dynamics: A first case study.
\newblock arXiv:2202.05073, 2022.

\bibitem{kuznetsov_tutorial}
Yu.~A. Kuznetsov.
\newblock Two-parameter bifurcation analysis of equilibria and limit cycles
  with matcont.
\newblock Tutorial {IV}.

\bibitem{MR1694631}
Yu.~A. Kuznetsov.
\newblock Numerical normalization techniques for all codim {$2$} bifurcations
  of equilibria in {ODE}'s.
\newblock {\em SIAM J. Numer. Anal.}, 36(4):1104--1124, 1999.

\bibitem{MR2071006}
Yuri~A. Kuznetsov.
\newblock {\em Elements of applied bifurcation theory}, volume 112 of {\em
  Applied Mathematical Sciences}.
\newblock Springer-Verlag, New York, third edition, 2004.

\bibitem{pnas_cancer}
Jiyoung Lee, Jinho Lee, Kevin~S. Farquhar, Jieun Yun, Casey~A. Frankenberger,
  Elena Bevilacqua, Kam Yeung, Eun-Jin Kim, G{\'a}bor Bal{\'a}zsi, and
  Marsha~Rich Rosner.
\newblock Network of mutually repressive metastasis regulators can promote cell
  heterogeneity and metastatic transitions.
\newblock {\em PNAS}, 111(3):E364--73, 2014.

\bibitem{MR3542954}
Jean-Philippe Lessard.
\newblock Rigorous verification of saddle-node bifurcations in {ODE}s.
\newblock {\em Indag. Math. (N.S.)}, 27(4):1013--1026, 2016.

\bibitem{MR3808252}
Jean-Philippe Lessard, Evelyn Sander, and Thomas Wanner.
\newblock Rigorous continuation of bifurcation points in the diblock copolymer
  equation.
\newblock {\em J. Comput. Dyn.}, 4(1-2):71--118, 2017.

\bibitem{MR2319947}
Teruya Minamoto and Mitsuhiro~T. Nakao.
\newblock Numerical method for verifying the existence and local uniqueness of
  a double turning point for a radially symmetric solution of the perturbed
  {G}elfand equation.
\newblock {\em J. Comput. Appl. Math.}, 202(2):177--185, 2007.

\bibitem{OT_cusp}
Livia Owen and Johan~Matheus Tuwankotta.
\newblock Computation of fold and cusp bifurcation points in a system of
  ordinary differential equations using the {L}agrange multiplier method.
\newblock {\em Int. J. Dyn. Control}, 10(2):363--376, 2022.

\bibitem{H1881}
H.~Poincar{\'e}.
\newblock M{\'e}moire sur les courbes d{\'e}finies par une {\'e}quation
  diff{\'e}rentielle (i).
\newblock {\em Journal de Math{\'e}matiques Pures et Appliqu{\'e}es},
  7:375--422, 1881.

\bibitem{MR787206}
Dirk Roose and Robert Piessens.
\newblock Numerical computation of nonsimple turning points and cusps.
\newblock {\em Numer. Math.}, 46(2):189--211, 1985.

\bibitem{MR2059468}
Ken'ichiro Tanaka, Sunao Murashige, and Shin'ichi Oishi.
\newblock On necessary and sufficient conditions for numerical verification of
  double turning points.
\newblock {\em Numer. Math.}, 97(3):537--554, 2004.

\bibitem{MR4043243}
Johan~Matheus Tuwankotta and Eric Harjanto.
\newblock Strange attractors in a predator-prey system with non-monotonic
  response function and periodic perturbation.
\newblock {\em J. Comput. Dyn.}, 6(2):469--483, 2019.

\bibitem{MR4664056}
Jan~Bouwe van~den Berg, Olivier H\'{e}not, and Jean-Philippe Lessard.
\newblock Constructive proofs for localised radial solutions of semilinear
  elliptic systems on {$\Bbb R^d$}.
\newblock {\em Nonlinearity}, 36(12):6476--6512, 2023.

\bibitem{jb_center_manifold2}
Jan~Bouwe van~den Berg, Wouter Hetebrij, and Bob Rink.
\newblock More on the parameterization method for center manifolds.
\newblock arXiv:2003.00701, 2020.

\bibitem{jb_center_manifold1}
Jan~Bouwe van~den Berg, Wouter Hetebrij, and Bob Rink.
\newblock The parameterization method for center manifolds.
\newblock {\em J. Differential Equations}, 269(3):2132--2184, 2020.

\bibitem{MR3779642}
Jan~Bouwe van~den Berg and Jonathan Jaquette.
\newblock A proof of {W}right's conjecture.
\newblock {\em J. Differential Equations}, 264(12):7412--7462, 2018.

\bibitem{MR4238006}
Jan~Bouwe van~den Berg, Jean-Philippe Lessard, and Elena Queirolo.
\newblock Rigorous verification of {H}opf bifurcations via desingularization
  and continuation.
\newblock {\em SIAM J. Appl. Dyn. Syst.}, 20(2):573--607, 2021.

\bibitem{MR4372114}
Jan~Bouwe van~den Berg and Elena Queirolo.
\newblock Rigorous validation of a {H}opf bifurcation in the
  {K}uramoto-{S}ivashinsky {PDE}.
\newblock {\em Commun. Nonlinear Sci. Numer. Simul.}, 108:Paper No. 106133,
  2022.

\bibitem{MR3792794}
Thomas Wanner.
\newblock Computer-assisted bifurcation diagram validation and applications in
  materials science.
\newblock In {\em Rigorous numerics in dynamics}, volume~74 of {\em Proc.
  Sympos. Appl. Math.}, pages 123--174. Amer. Math. Soc., Providence, RI, 2018.

\bibitem{MR2534406}
Daniel Wilczak and Piotr Zgliczy{\'n}ski.
\newblock Period doubling in the {R}\"ossler system---a computer assisted
  proof.
\newblock {\em Found. Comput. Math.}, 9(5):611--649, 2009.

\bibitem{MR1639986}
Nobito Yamamoto.
\newblock A numerical verification method for solutions of boundary value
  problems with local uniqueness by {B}anach's fixed-point theorem.
\newblock {\em SIAM J. Numer. Anal.}, 35(5):2004--2013 (electronic), 1998.

\bibitem{MR3390404}
Piotr Zgliczy{\'n}ski.
\newblock Steady state bifurcations for the {K}uramoto-{S}ivashinsky equation:
  a computer assisted proof.
\newblock {\em J. Comput. Dyn.}, 2(1):95--142, 2015.

\bibitem{MR1951954}
Huaiping Zhu, Sue~Ann Campbell, and Gail S.~K. Wolkowicz.
\newblock Bifurcation analysis of a predator-prey system with nonmonotonic
  functional response.
\newblock {\em SIAM J. Appl. Math.}, 63(2):636--682, 2002.

\end{thebibliography}

\end{document}